\documentclass[11pt,twoside,leqno,a4paper,final]{amsart}
\usepackage[a4paper=true,pdfpagelabels,colorlinks]{hyperref} 
\usepackage{amsmath}
\usepackage{mathtools,nccmath}

\usepackage{amssymb}
\usepackage{enumitem}

\usepackage[dvipsnames]{xcolor}

\usepackage{mathrsfs} 

\newcommand{\vertiii}[1]{{\left\vert\kern-0.25ex\left\vert\kern-0.25ex\left\vert #1 
    \right\vert\kern-0.25ex\right\vert\kern-0.25ex\right\vert}}

\theoremstyle{plain}
\newtheorem{proposition}{Proposition}[section]
\newtheorem{corollary}[proposition]{Corollary}
\newtheorem{lemma}[proposition]{Lemma}
\newtheorem{theorem}[proposition]{Theorem}

\theoremstyle{definition}

\newtheorem{remark}[proposition]{Remark}

\newcommand{\B}{\mathscr{B}}
\newcommand{\BB}{\mathcal{B}}

\newcommand{\D}{{\mathbb D}}

\renewcommand{\H}{\mathcal{H}}

\newcommand{\N}{{\mathbb N}}

\def\Ap2{A^p_{\omega_p}}
\def\Lp2{L^p_{\omega_p}}

\def\om{\omega}

\newcommand{\R}{{\mathbb R}}
\newcommand{\Q}{{\mathbb Q}}

\def\om{\omega}

\def\SW{\mathcal{SW}}

\def\b{\beta}

\usepackage{bm}

\makeatletter
\newcommand{\opnorm}{\@ifstar\@opnorms\@opnorm}
\newcommand{\@opnorms}[1]{%
  \left|\mkern-1.5mu\left|\mkern-1.5mu\left|
   #1
  \right|\mkern-1.5mu\right|\mkern-1.5mu\right|
}
\newcommand{\@opnorm}[2][]{%
  \mathopen{#1|\mkern-1.5mu#1|\mkern-1.5mu#1|}
  #2
  \mathclose{#1|\mkern-1.5mu#1|\mkern-1.5mu#1|}
}
\makeatother

\makeatletter
\newcommand{\boldopnorm}{\@ifstar\@boldopnorms\@boldopnorm}
\newcommand{\@boldopnorms}[1]{%
  \pmb{\left|}\mkern-1.5mu\pmb{\left|}\mkern-1.5mu\pmb{\left|}
   #1
  \pmb{\right|}\mkern-1.5mu\pmb{\right|}\mkern-1.5mu\pmb{\right|}
}
\newcommand{\@boldopnorm}[2][]{%
  \mathopen{#1\pmb{|}\mkern-1.5mu#1\pmb{|}\mkern-1.5mu#1\pmb{|}}
  #2
  \mathclose{#1\pmb{|}\mkern-1.5mu#1\pmb{|}\mkern-1.5mu#1\pmb{|}}
}
\makeatother

\numberwithin{equation}{section}

\theoremstyle{plain} 
\newcommand{\thistheoremname}{}
\newtheorem*{genericthm*}{\thistheoremname}
\newenvironment{namedthm*}[1]
{\renewcommand{\thistheoremname}{#1}%
	\begin{genericthm*}}
	{\end{genericthm*}}

\theoremstyle{definition} 
\newcommand{\thisdefinitionname}{}
\newtheorem*{genericdefinition*}{\thisdefinitionname}
\newenvironment{nameddefinition*}[1]
{\renewcommand{\thisdefinitionname}{#1}%
	\begin{genericdefinition*}}
	{\end{genericdefinition*}}

\makeatletter
\@namedef{subjclassname@2020}{\textup{2020} Mathematics Subject Classification}
\makeatother

\begin{document}

\title [Words of analytic paraproducts 
on Bergman spaces]
{Words of analytic paraproducts
on Bergman spaces induced by smooth rapidly decreasing weights}

\author[C. Cascante]{Carme Cascante}
\address{C. Cascante: Departament de Matem\`atiques i
    Inform\`atica, Universitat  de Barcelona,
     Gran Via 585, 08071 Barcelona, Spain}
\email{cascante@ub.edu}
\author[J. F\`abrega]{Joan F\`abrega}
\address{J. F\`abrega: Departament de Matem\`atiques i
    Inform\`atica, Universitat  de Barcelona,
     Gran Via 585, 08071 Barcelona, Spain}
\email{joan$_{-}$fabrega@ub.edu}
\author[D. Pascuas]{Daniel Pascuas}
\address{D. Pascuas: Departament de Matem\`atiques i
    Inform\`atica, Universitat  de Barcelona,
     Gran Via 585, 08071 Barcelona, Spain}
\email{daniel$_{-}$pascuas@ub.edu}
\author[J. A. Pel\'aez]{Jos\'e \'Angel Pel\'aez}
\address{J. A. Pel\'aez: Departamento de An\'alisis Matem\'atico, Universidad de M\'alaga, Campus de
Teatinos, 29071 M\'alaga, Spain} \email{japelaez@uma.es}

\thanks{
The research of the first, second, and third authors
	was supported in part by
        Ministerio de Ciencia e Innovaci\'{o}n, 
        Spain, project  PID2021-123405NB-I00.  The first author is also supported by the Spanish State Research Agency, through the Severo Ochoa and Mar\'{\i}a de Maeztu Program for Centers and Units of Excellence in R\&D 
(CEX2020-001084-M), and by Departament de Recerca i Universitats  grant SGR 2021-00087.
The research of the fourth author was supported in part by Ministerio de Ciencia e Innovaci\'on, Spain, project  PID2022-136619NB-I00; La Junta de Andaluc{\'i}a,
project FQM210.}
\date{\today}



\date{\today}

\begin{abstract} 
For a fixed analytic function g on the unit disc, we consider the analytic paraproducts induced by g, which are formally defined by 
$T_gf(z)=\int_0^zf(\zeta)g'(\zeta)d\zeta$,  $S_gf(z)=\int_0^zf'(\zeta)g(\zeta)d\zeta$, and $M_gf(z)=g(z)f(z)$.   An  $N$-letter $g$-word is an operator of the form $L=L_1\cdots L_N$, where each $L_j$ is either $M_g$, $S_g$ or $T_g$. 
It has been recently proved in \cite{Aleman:Cascante:Fabrega:Pascuas:Pelaez,Aleman:Cascante:Fabrega:Pascuas:Pelaez2} that understanding the boundedness of  a $g$-word on classical Hardy and Bergman spaces is a challenging problem due to the potential cancellations involved.
Our main result provides a complete quantitative characterization of the boundedness of an arbitrary $g$-word on a  weighted Bergman space  $A^p_{\om^{p/2}}$,  where $\om=e^{-2\varphi}$ is a smooth rapidly decreasing weight. In particular, it
 states that any   $N$-letter $g$-word   such that $\#\{j:L_j=T_g\}=n\ge 1$ is bounded on $A^p_{\om^{p/2}}$
if and only if $g$ satisfies the "fractional" Bloch-type condition 
\[
\|g\|_{\B^s_\varphi}^s:=
\sup_{z\in\D}\frac{s|g(z)|^{s-1}|g'(z)|}{1+\varphi'(|z|)}<\infty,
\]
 where $s=\frac{N}{n}$, and  $\|L\|_{A^p_{\om^{p/2}}}\simeq \|g\|_{\B^s_\varphi}^N$.

The class of smooth rapidly decreasing weights contains
 the radial weights
\begin{equation*}
\omega_n(z)=e^{-2\exp_{n}(g_{\alpha,c}(|z|))},
\quad\mbox{where}\quad g_{\alpha,c}(r)=\tfrac{c}{(1-r^2)^{\alpha}},
\quad\mbox{for $c,\alpha>0$,}
\end{equation*} 
$\exp_0(x)=x$ and  $\exp_n(x)=e^{\exp_{n-1}(x)}$, for $n\in\N$. Therefore it contains
 weights which decrease arbitrarily rapidly  to zero as $|z|\to 1^-$. 

\end{abstract}
\maketitle

\noindent
\section{Introduction}
Let $\H(\D)$ denote the space of analytic functions on the unit disc $\D$ of the complex plane.  
For a non-negative function $\om\in L^1([0,1))$ such that  $\int_r^1 \omega(s)\,ds>0$ for any $r\in [0,1)$,
the extension to $\D$ defined by $\om(z)=\om(|z|)$ 
 is called a radial weight.  For $0<p<\infty$ and a radial weight $\omega$, the weighted Bergman
space $A^p_\omega$ consists of those $f\in\H(\D)$ for which
    $$
    \|f\|_{A^p_\omega}^p=\int_\D|f(z)|^p\omega(z)\,dA(z)<\infty,\index{$\Vert\cdot\Vert_{A^p_\omega}$}
    $$
where $dA(z)=\frac{dx\,dy}{\pi}$ is the normalized Lebesgue area measure on $\D$. The condition
$\int_r^1 \omega(s)\,ds>0$,  $r\in [0,1)$,  is justified by the fact that $A^p_\omega=\H(\D)$ whenever it does not hold.
As usual,  we write $A^p_\alpha$ for the Bergman space induced by the  standard weight 
$\omega(z)=(\alpha+1)(1-|z|^2)^\alpha$, $\alpha>-1$.  

For any $g\in\H(\D)$, we consider the {\em$g$-analytic paraproducts} $M_g,T_g,S_g:\H(\D)\to\H(\D)$ defined by
$M_gf= fg$,
\[
T_gf(z)= \int_0^z f(\zeta)g'(\zeta)\,d\zeta 
\quad\mbox{and}\quad
S_gf(z)= \int_0^z f'(\zeta)g(\zeta)\,d\zeta.
\]
The boundedness  of $g$-analytic paraproducts has been studied on many spaces of analytic functions since the seminal papers  \cite{AC,AS,AS2,Duren:Romberg:Shields}, where the authors described  
their action on
classical Hardy spaces 
and standard Bergman spaces. 
Going further, motivated by understanding meaningful cancelation phenomena, it has been recently considered 
the  boundedness of compositions (products) of analytic paraproducts acting on them \cite{Aleman:Cascante:Fabrega:Pascuas:Pelaez,Aleman:Cascante:Fabrega:Pascuas:Pelaez2}.
We recall that  
 an {\em $N$-letter $g$-word} is an operator of the form $L=L_1\cdots L_N$, where each $L_j$ is either $M_g$, $S_g$ or $T_g$. 
 In  \cite{Aleman:Cascante:Fabrega:Pascuas:Pelaez2},
among other results, it is obtained a characterization of the symbols $g$ such that an $N$-letter $g$-word is bounded on $H^p$  and on $A^p_\alpha$, which only depends on the number of appearances of each of the letters
$T_g$, $S_g$ and $M_g$ in the given word.

In this paper we are interested in describing  the boundedness of
$N$-letter $g$-words in the setting of Bergman spaces $A^p_{\om_p}$, $\om_p=\om^{p/2}$, where $\om$ is a smooth rapidly decreasing weight. 
A radial weight $\omega$ is {\em smooth rapidly decreasing} if $\omega=e^{-2\varphi}$, where $\varphi$  satisfies the following conditions:

\begin{enumerate}[label={\sf(\alph*)}, topsep=3pt, leftmargin=32pt,itemsep=3pt] 
	
	\item\label{rapidly:decreasing:condition:a} 
	$\varphi$ is a radial positive $C^2$ function on $\D$ which is increasing on $[0,1)$ and satisfies that
	$\Delta\varphi>0$ on $\D$, where $\Delta$ denotes the standard Laplace operator.
	
	\item\label{rapidly:decreasing:condition:b}  $\left(\Delta\varphi(z)\right)^{-1/2}\simeq\tau(|z|)$, where
	$\tau$ is a positive decreasing $C^1$ function on $[0,1)$
	such that $\lim_{r\to 1^-}\tau(r)=\lim_{r\to 1^-}\tau'(r)=0$.
	
	\item\label{rapidly:decreasing:condition:c}
	There exists  a constant $C>0$  such that either
	$\tau(r)(1-r)^{-C}$ increases for $r$ close to $1$ or
	\begin{equation}\label{eqn:condition:tauprime}
	\lim_{r\to 1^-}\tau'(r)\log\frac{1}{\tau(r)}=0.
	\end{equation}
	\item\label{rapidly:decreasing:condition:e}
	$\varphi$ is convex on $[0,1)$ and there is $\eta>0$ such that the function $(1+\varphi') \tau^\eta$ is essentially decreasing in $[0,1)$. Recall that a function $h:\,[0,1)\to \mathbb{R}$ is {\em essentially decreasing} when
	$h(s)\lesssim h(t)$, for $0\le t\le s<1$.
\end{enumerate}

Here, as usual, for two non-negative functions $A$ and $B$, 
$A\lesssim B$ ($B\gtrsim A$) means that there is a finite positive constant $C$, independent of the variables involved, which satisfies $A\le C\,B$. Moreover, we write $A\simeq B$  when $A\lesssim B$ and $B\lesssim A$.

The class of smooth rapidly decreasing weights is denoted by $\SW$.
Obviously it is enough that  the radial positive $C^2$ function $\varphi$ on $\D$ satisfies the four conditions \ref{rapidly:decreasing:condition:a}-\ref{rapidly:decreasing:condition:e} on an interval $[r_0,1)$, for some $r_0\in [0,1)$, because then there exists $\nu\in \SW$ such that $A^p_{\om^p}=A^p_{\nu^p}$, $0<p<\infty$.
 We recall that, in the literature, the weights $\omega=e^{-2\varphi}$ satisfying the three conditions \ref{rapidly:decreasing:condition:a}-\ref{rapidly:decreasing:condition:c} are called {\em rapidly decreasing} (see \cite{Pau-Pelaez:JFA2010,CasFabPasPelJFA2024}).

\medskip

The class $\SW$ 
contains
 weights which decrease arbitrarily rapidly  to zero as $r\to 1^-$. Namely, if $\exp_0(x)=x$ and  $\exp_n(x)=e^{\exp_{n-1}(x)}$, for $n\in\N$,  then, 
for each $n\in \N_0$, the radial weight
\begin{equation}\label{eqn:examples:intro}
\omega_n(r)=e^{-2\exp_{n}(g_{\alpha,c}(r))},
\quad\mbox{where}\quad g_{\alpha,c}(r)=\tfrac{c}{(1-r^2)^{\alpha}},
\quad\mbox{for $c,\alpha>0$,}
\end{equation}
belongs to $\SW$ (see Corollary \ref{cor:examples}). We note that any of these weights satisfies condition \eqref{eqn:condition:tauprime}.

On the other hand, the
class  $\SW$ does not include the standard weights, but it
includes weights which decrease to zero slightly quicker than any standard weight $(\beta+1)(1-|z|^2)^\beta$, $\beta>0$. In fact,  
for each $\alpha>1$ there exists a radial positive $C^2$ function $\varphi_\alpha$ on $\D$ such that $\varphi_\alpha(r)=\left( \log\frac{e}{1-r^2}\right)^\alpha$ on  some interval
$[r_\alpha, 1)$, $r_\alpha\in [0,1)$, and
 \begin{displaymath}
\omega(z)=\exp
\left(-2\varphi_\alpha(|z|)\right)
\end{displaymath}
belongs to $\SW$.

\medskip

Let us recall that $S_g$ and $M_g$ are bounded on  $A^p_{\om_p}$ if and only if $g\in H^\infty$ and 
$\|M_g\|_{A^p_{\om_p}}\simeq  \|S_g\|_{A^p_{\om_p}}\simeq \|g\|_{H^\infty}$ \cite[Lemma~4.1]{CasFabPasPelJFA2024},
where as usual $H^\infty$ is the space of bounded analytic functions on $\D$.  
 As for the boundedness of the integration operator $T_g$, 
 by  \cite[(9.3)]{CP} (see also \cite{PavP})
for any $p,q\in(0,\infty)$ and $\omega=e^{-2\varphi}\in\SW$, we have
\begin{equation}\label{eq:LP2}
\|f\|^p_{A^p_{\om^q}}
\simeq |f(0)|^{p}+\|f'\|^p_{A^p_{\omega^{q} (1+\varphi')^{-p}}} \qquad(f\in\H(\D)).
\end{equation}
 This Littlewood-Paley type formula together with \eqref{eqn:tau:1+varphiprime:on:Ddeltaa} below,
allows to omit the hypotheses (6) in \cite[Theorem~2]{Pau-Pelaez:JFA2010} to obtain that $T_g\in\BB(A^p_{\omega_p})$ if and only if 
$g$ belongs to the Bloch-type space 
\[
\B_{\varphi}:=\left\{g \in\H(\D): \|g\|_{\B_\varphi}=\sup_{z\in\D}\frac{|g'(z)|}{1+\varphi'(|z|)}<\infty\right\}.
\]
 Moreover, $\|T_g\|_{A^p_{\omega_p}} \simeq \|g\|_{\B_{\varphi}}.$  

In order to deal with the case of $N$-letters $g$-words, $N\ge 2$, some definitions and notations are needed.
For a function $\psi:\D\to\R$ we define 
\begin{equation*}
	|\nabla\psi|(z)
	:=\limsup_{w\to z}\frac{|\psi(w)-\psi(z)|}{|w-z|}\in[0,\infty]
	\qquad(z\in\D).
\end{equation*}
The notation  $|\nabla\psi|(z)$ is justified by the fact that, when $\psi$ is differentiable at $z$, $|\nabla\psi|(z)$ is just the Euclidean norm $|\nabla\psi(z)|$ of the gradient of $\psi$ at $z$. 
In particular,  for $q\ge 1$ and $g\in\H(\D)$, we have that
$|\nabla|g|^q|=q|g|^{q-1}|g'|$.
Now, for any $q\ge1$  we introduce   new   Bloch-type classes of power functions 
$$\B_{\varphi}^q:=\{g\in\H(\D):\|g\|_{\B^q_\varphi}<\infty\},$$ where
\[
\|g\|_{\B^q_\varphi}^q:=
\sup_{z\in\D}\frac{|\nabla|g|^q|(z)}{1+\varphi'(|z|)}
\qquad(g\in\H(\D)).
\]
So  $\B_{\varphi}^1=\B_{\varphi}$ and, for $q\in \N$, we have that $g\in {\B^q_\varphi}$ if and only if $g^q\in \B_{\varphi}$.

Let $\N_0:=\N\cup\{0\}$, where, as usual, $\N$ denotes the set of positive integers. 
Let $\ell,m,n\in\N_0$ so that $N=\ell+m+n\ge 1$.
Then $W_g(\ell,m,n)$ is the set of all $g$-words of the form 
\begin{equation*}
L=L_1\cdots L_N,
\end{equation*} 
with $\#\{j:L_j=M_g\}=\ell$, $\#\{j:L_j=S_g\}=m$, and 
$\#\{j:L_j=T_g\}=n$. 

Throughout the manuscript  the space of bounded  linear operators on $A^p_\omega$ is denoted by $\BB(A_\omega^p)$, and for  any linear map  $L:\H(\D)\to\H(\D)$ we write
$\|L\|_{A_\omega^p}:=\sup \{\|Lf\|_{A_\omega^p}:\|f\|_{A_\omega^p}=1\}$. We refer to this quantity as the operator norm of $L$ on $A^p_\omega$, despite $A^p_\omega$ is not a normed space for $0 <p <1$.
 For any linear map $L:\H(\D)\to\H(\D)$, let
\[
\opnorm{L}_{A^p_{\omega_p}}
:=\sup\{\|Lf\|_{A^p_{\omega_p}}:f\in A^p_{\omega_p}(0),\,\|f\|_{A^p_{\omega_p}}\le1\},
\] 
where $A^p_{\omega_p}(0):=A^p_{\omega_p}\cap\H_0(\D)$ and 
$\H_0(\D):=\{f\in \H(\D):\,f(0)=0\}$.

Our main result provides a complete characterization of bounded $N$-letter $g$-words on $A^p_{\omega_{p}}$, for $\omega\in\SW$.

\begin{theorem}
\label{thm:sharp:estimate:norm:word}
Let be $\om=e^{-2\varphi}\in \SW$, $0<p<\infty$, $g\in\H(\D)$ and 
$L_g\in W_g(\ell,m,n)$, where  $\ell,m,n\in\N_0$ and $N=\ell+m+n\ge1$.
	\begin{enumerate}[label={\sf\alph*)},topsep=3pt, 
		leftmargin=0pt, itemsep=4pt, wide, listparindent=0pt, itemindent=6pt] 
\item \label{thm:sharp:estimate:norm:word:a}
If $n=0$, then $L_g$ is bounded on either $\Ap2$ or $\Ap2(0)$
if and only if $g\in H^\infty$. Moreover,
\begin{equation}
\label{eqn:sharp:estimate:norm:wordceroTg}
\|L_g\|_{\Ap2}\simeq\opnorm{L_g}_{\Ap2}
\simeq\|g\|^{N}_{H^\infty}.
\end{equation}
\item \label{thm:sharp:estimate:norm:word:b}
If $n\ge1$, then $L_g$ is bounded on either $\Ap2$ or $\Ap2(0)$ if and only $g\in \B_\varphi^s$, where $s=\frac{N}n=\frac{\ell+m}n+1$. Moreover,
\begin{equation}
\label{eqn:sharp:estimate:norm:word}
\|L_g\|_{\Ap2}\simeq\opnorm{L_g}_{\Ap2}
\simeq\|g\|^{N}_{\B_\varphi^s}.
\end{equation}
\end{enumerate}
 \end{theorem}

As for the proof of Theorem~\ref{thm:sharp:estimate:norm:word}~\ref{thm:sharp:estimate:norm:word:a},  the identity $\|L_g\|_{\Ap2}
\simeq\|g\|^{N}_{H^\infty}$  has been proved in \cite[Theorem~1.3~b)]{CasFabPasPelJFA2024} for any radial weight $\omega$. The remaning inequality
$\|g\|^{N}_{H^\infty}\lesssim \opnorm{L_g}_{\Ap2}$  of \eqref{eqn:sharp:estimate:norm:wordceroTg}
can be proved bearing in mind \cite[Theorem~1.3~a)]{CasFabPasPelJFA2024}  (see also Theorem~\ref{thm:compo} below) and mimicking its proof.

In \cite[Theorems~1.3-1.4]{CasFabPasPelJFA2024} it is proved that
 when $s=\frac{N}n\in\N$ we have  $\|L_g\|_{\Ap2}\simeq\|S_g^{s-1}T_g\|_{\Ap2}^n\simeq\|g^s\|_{\B_\varphi}^n=\|g\|_{\B_\varphi^s}^N$, for any radial weight $\omega$. 
Therefore Theorem~\ref{thm:sharp:estimate:norm:word}~\ref{thm:sharp:estimate:norm:word:b} holds for any radial weight $\omega$ when either $\ell+m=0$  or $N\in \{1,2\}$. In the remaining cases, a good number of significant obstacles which require  new  techniques and ideas have to be overcome in order to complete the proof.

The first one is related with the following embeddings among the classes of symbols $\B^{s}_\varphi$.
\begin{theorem}\label{prop:radicality:estimateintro}
Let be $\om=e^{-2\varphi}\in \SW$. Then, 
\begin{equation}\label{eq:i1n}
\B^{q_2}_\varphi\subset \B^{q_1}_\varphi, \quad\text{for any $1\le q_1<q_2$.}
\end{equation}
In addition, we have the estimate 
\begin{equation}\label{eq:i1}
\|g\|_{\B^{q_1}_\varphi}\lesssim \|g\|_{\B^{q_2}_\varphi}
\qquad(g\in\H(\D)).
\end{equation}
\end{theorem}
Estimate \eqref{eq:i1} is essential to prove Theorem~\ref{thm:sharp:estimate:norm:word}~\ref{thm:sharp:estimate:norm:word:b}. Both results can be used to study the boundedness
of some linear combinations of $g$-words, for instance:

\begin{corollary}\label{cor:main:thm}
Let be $\om=e^{-2\varphi}\in \SW$, $0<p<\infty$, $g\in\H(\D)$ .
Let $L_{g,j}\in W_g(\ell_j,m_j,n_j)$, $N_j=\ell_j+m_j+n_j$, and $s_j=\frac{N_j}{n_j}$, with $n_j>0$, for $j\in\{0,\dots,J\}$.
If 
$N_j<N_0$ and $s_j<s_0$, for $j\in\{1,\dots,J\}$, then $L_g=L_{g,0}+\cdots+L_{g,J}$ is bounded on $\Ap2$ if and only if so is $L_{g,0}$.
\end{corollary}

In \cite[Section~2]{Aleman:Cascante:Fabrega:Pascuas:Pelaez2}, by using their conformally invariance, it is proved that 
BMOA-type classes of symbols which appear in the context of Hardy spaces and standard Bergman spaces satisfy analogous embeddings to \eqref{eq:i1n}. 
 Due to the different nature of Hardy spaces (and standard Bergman spaces) and weighted Bergman spaces induced by smooth rapidly decreasing weights,
the spaces of symbols  $\B^{s}_\varphi$ are strictly contained in the classical Bloch space, so they  are not conformally invariant.  This fact leads us to use different skills in the proof of Theorem~\ref{prop:radicality:estimateintro}. Indeed,
we provide a direct  proof of Theorem~\ref{prop:radicality:estimateintro} where it is strongly used the identity (see \eqref{eqn:Bqphi:Lipschitz} below)
\[
\|g\|_{\B^q_\varphi}^q
=\sup_{\substack{z,w\in\D\\ z\ne w}}
\frac{\bigl||g(z)|^q-|g(w)|^q\bigr|}{\beta_{\varphi}(z,w)}
\qquad(g\in\H(\D),q\ge1),
\]
where $\beta_{\varphi}(z,w)$ is  the distance  on $\D$ induced by the Riemannian metric
\[
\tfrac12\bigl(1+\varphi'(|z|)\bigr)(dz\otimes d\overline{z}).
\]
It is worth mentioning that, when $\varphi$ satisfies an additional condition (see \eqref{varphi}), the Bloch-type class $\B^q_\varphi$ coincides with the growth class 
\begin{equation}\label{eqn:def:Hintyqvarphi}
H^{\infty,q}_\varphi:=
\left\{g \in\H(\D): \|g\|^q_{H^{\infty,q}_\varphi}
                      =\sup_{z\in\D}\frac{|g(z)|^q}{\varphi(|z|)}<\infty\right\}
\end{equation}
and therefore  \eqref{eq:i1n}  trivially holds because
if $1\le q_1<q_2$ then
\[\|g\|_{H^{\infty,q_1}_\varphi}\lesssim  \|g\|_{H^{\infty,q_2}_\varphi}\qquad(g\in\H(\D)).
\]
However, $\|g\|_{\B^q_{\varphi}}$ and $\|g\|_{H^{\infty,q}_\varphi}$ are not comparable for any $g\in\H(\D)$, so 
 the inequality \eqref{eq:i1}  which
plays a key role in the proof 
of  Theorem~\ref{thm:sharp:estimate:norm:word}~\ref{thm:sharp:estimate:norm:word:b}, does not follow from the above inequality.

Especially, if $q_1,q_2\in \N$, Theorem~\ref{prop:radicality:estimateintro} states that  $\B_\varphi$ has the radicality property \cite{CasFabPasPelJFA2024}. To figure out the radicality property for a class of analytic functions may be a tough problem which has attracted some attention in the recent years \cite{Aleman:Cascante:Fabrega:Pascuas:Pelaez2,CasFabPasPelJFA2024,Gomez-Lefevre-Queffelec-2025}. 
 The classical Bloch space satisfies the radicality property \cite[Section~2]{Aleman:Cascante:Fabrega:Pascuas:Pelaez}, therefore when $\mathcal{T}(A^p_\omega)=\left \{g\in \H(\D): T_g\,\text{ is bounded on}\, A^p_\omega\right\}$ coincide with it, this latter space has the radicality property. For instance, it happens if $\omega$ is a radial doubling weight or a Bekoll\'e-Bonami weight, see   \cite{CasFabPasPelJFA2024} and the references therein. In addition,  by using  operator theoretical arguments,  it has been recently proved in   \cite[Theorem~1.1]{CasFabPasPelJFA2024} that
$\mathcal{T}(A^p_\omega)=\left \{g\in \H(\D): T_g\,\text{ is bounded on}\, A^p_\omega\right\}$ has the radicality property for any
 radial weight  $\omega$ and $0<p<\infty$.
\vspace{1em}

The second major obstacle tackled in the proof of Theorem~\ref{thm:sharp:estimate:norm:word} consists on proving the following inequality
 (see Proposition~\ref{lem:sharp:lower:estimate:norm:word:3} below)
\begin{equation}
\label{eqn:lem3:sharp:lower:estimate:norm:wordintro}
\|g\|^{\sigma +1}_{\B^{\sigma+1}_{\varphi}}\lesssim
\opnorm{Q_g^{\sigma,1}}_{\Lp2} 
\qquad(g\in\H(\overline{\D})), \quad 0<\sigma\in\Q,
\end{equation}
which involves the intermediate operators  $Q_g^{\sigma,\ell} f= |g|^{\sigma \ell} T^{\ell}_g f$, $\ell\in \N$, $\sigma>0$.

Our proof of \eqref{eqn:lem3:sharp:lower:estimate:norm:wordintro} has a several complex variable flavour. Indeed,  
it is based on a representation formula,   
 a convenient application of Stokes' theorem, new norm estimates for intermediate operators  $Q_g^{\sigma,\ell}$,  pointwise norm estimates
 and
 precise norm estimates  of the Bergman reproducing kernel $K^a_\om$ of $A^2_\omega$,  on weighted Bergman spaces induced by perturbations of $\omega_p$. However, the
analogous inequality in the Hardy setting (see \cite[Proposition 7.4]{Aleman:Cascante:Fabrega:Pascuas:Pelaez2}) is proved by a completely different method 
which uses the theory of tent spaces and a description of the space of symbols in terms of Carleson measures. 

The paper is organized as follows. 
In Section \ref{section:auxiliary:results} we gather auxiliary results which will be used in proving our main theorem. Section \ref{section:symbols} is devoted to study our Bloch classes of power functions, and, in particular, we prove Theorem \ref{prop:radicality:estimateintro}. In Section \ref{section:reduction:to:SmTn} we reduce the proof of our main result to the case $L_g=S^m_gT^n_g$. We finish its proof 
in Sections \ref{section:proof:of:proposition4.1} and \ref{section:proof:of:proposition4.2}. 
In Section \ref{section:proof:of:proposition4.2} we also prove 
Corollary \ref{cor:main:thm}.
Finally, Section \ref{examples} addresses the proof that the weights \eqref{eqn:examples:intro} belong to $\SW$.

\section{Auxiliary results}\label{section:auxiliary:results}

\subsection{Properties of the function $\varphi$}

Next proposition collects some well known properties of (one-half of) the logarithm $\varphi$ of a rapidly decreasing weight (see, for example, \cite[Lemma 32]{CP} and \cite[Lemma 2.1]{Pau-Pelaez:JFA2010}). Due to their simplicity, and for the sake of completeness, we give a full proof of them. 

\begin{proposition}\label{prop:properties:varphi}
Let $\varphi$ be a function satisfying \ref{rapidly:decreasing:condition:a} and \ref{rapidly:decreasing:condition:b}. Then:
\begin{enumerate}[label={\sf\alph*)},topsep=3pt, 
		leftmargin=0pt, itemsep=4pt, wide, listparindent=0pt, itemindent=6pt] 
\item\label{item1:properties:varphi}
 ${\displaystyle\lim_{r\to1^{-}}\frac{\tau(r)}{1-r}=0}$.
\item\label{item2:properties:varphi} 
$|\tau(r)-\tau(s)|\lesssim|r-s|$, for $r,s\in[0,1)$.
\item\label{item3:properties:varphi}
 ${\displaystyle\lim_{r\to 1^-}(1-r)\varphi'(r)=\infty}$, and, in particular,
${\displaystyle\lim_{r\to 1^-}\varphi'(r)=\infty}$.
\item\label{item4:properties:varphi} 
${\displaystyle
\lim_{r\to 1^-}\frac{\varphi(r)}{\log(\frac1{1-r})}=\infty}$, and, in particular,
${\displaystyle\lim_{r\to 1^-}\varphi(r)=\infty}$.
\item\label{item5:properties:varphi}
 ${\displaystyle\lim_{r\to1^-}\tau(r)\varphi'(r)=\infty}$, or, equivalently, 
         ${\displaystyle
          \lim_{r\to 1^-}\frac{\varphi''(r)}{(\varphi'(r))^2}=0}$.
\item\label{item6:properties:varphi} There exists $0<\delta<\inf\{(1-|z|)/\tau(z):z\in\D\}$ such that
\begin{equation}\label{eqn:tau:1+varphiprime:on:Ddeltaa}
\tau(z)\simeq\tau(a)
\quad\mbox{and}\quad
1+\varphi'(|a|)\simeq1+\varphi'(|z|)
\quad(a\in\D,\,z\in D_{\delta}(a)),
\end{equation}
  where $D_{\delta}(a):=D(a,\delta\tau(a))$.
\end{enumerate}
\end{proposition}

\begin{proof}
Assertion \ref{item1:properties:varphi} directly follows from  L'H\^{o}pital's rule and \ref{rapidly:decreasing:condition:b}. Moreover,
it is clear that  \ref{rapidly:decreasing:condition:b} implies that  $\sup_{0<r<1}|\tau'(r)|<\infty$, so \ref{item2:properties:varphi} holds by the mean value theorem. 

Since 
$s\,\Delta\varphi(s)=\frac{d\,}{ds}\bigl(s\,\varphi'(s)\bigr)$, for any $0<s<1$, we have that
\begin{equation}\label{eqn:lemma:properties:varphi:0}
r\,\varphi'(r)=\int_0^rs\,\Delta\varphi(s)\,ds\qquad(0<r<1),
\end{equation}
and, taking into account \ref{rapidly:decreasing:condition:b}, we get that
\begin{equation*}
r\,(1-r)\,\varphi'(r)\gtrsim\frac{\int_0^r\frac{s}{\tau(s)^2}\,ds}{\frac1{1-r}}
=\Phi(r)\qquad(0<r<1).
\end{equation*}
But L'H\^{o}pital's rule and \ref{item1:properties:varphi} show that $\Phi(r)\to\infty$, as $r\to1^{-}$, so \ref{item3:properties:varphi} holds. Then \ref{item4:properties:varphi} directly follows from \ref{item3:properties:varphi}, by L'H\^{o}pital's rule again.

In order to prove \ref{item5:properties:varphi}, multiply both terms of \eqref{eqn:lemma:properties:varphi:0} by $\tau(r)$ to obtain that
\begin{equation*}
r\,\tau(r)\,\varphi'(r)
\gtrsim\frac{\int_0^r\frac{s}{\tau(s)^2}\,ds}{\frac1{\tau(r)}}=\Psi(r)
\qquad(0<r<1).
\end{equation*}
Since $\Psi(r)\to\infty$, as $r\to1^{-}$, by L'H\^{o}pital's rule and \ref{rapidly:decreasing:condition:b}, we have that $\lim_{r\to1^-}\tau(r)\varphi'(r)=\infty$, which is equivalent 
to $\lim_{r\to 1^-}\frac{\varphi''(r)}{(\varphi'(r))^2}=0$, because
\begin{equation*}
\bigl(\tau(r)\varphi'(r)\bigr)^{-2}\simeq\frac{\varphi''(r)}{(\varphi'(r))^2}+\frac1{r\varphi'(r)}
\qquad(0<r<1)
\end{equation*}
and $\lim_{r\to1^{-}}r\varphi'(r)=\infty$, by \ref{item3:properties:varphi}. Therefore \ref{item5:properties:varphi} holds.

Finally, we are going to prove \ref{item6:properties:varphi}. 
By \ref{item2:properties:varphi}, there is a positive constant~{$C$} such that $|\tau(r)-\tau(s)|\le C\,|r-s|$, for all $r,s\in[0,1)$.
Let 
$\delta_0=\frac12\min\bigl\{\frac1{C},\,{\displaystyle\inf_{z\in\D}}\frac{1-|z|}{\tau(z)}\bigr\}$,
 which is a positive number by \ref{item1:properties:varphi} and the hypothesis that $\tau$ is a radial positive continuous function on $\D$.  Then
\begin{equation*}
|\tau(z)-\tau(a)|\le C||z|-|a||\le C|z-a|\le C\delta_0\tau(a)\le\tfrac12\,\tau(a),
\end{equation*}
so $\frac12\tau(a)\le\tau(z)\le\frac32\tau(a)$,
 for any $a\in\D$ and $z\in D_{\delta_0}(a)$, and that shows the first 
estimate of \ref{item6:properties:varphi}. Now let us prove that 
\begin{equation}\label{eqn:lemma:properties:varphi:1}
0<\varliminf_{|a|\to1^{-}}
\inf_{z\in D_{\delta_0}(a)}\tfrac{\varphi'(|z|)}{\varphi'(|a|)}
\quad\mbox{and}\quad
\varlimsup_{|a|\to1^{-}}
\sup_{z\in D_{\delta_0}(a)}\tfrac{\varphi'(|z|)}{\varphi'(|a|)}<\infty.
\end{equation}
By \eqref{eqn:lemma:properties:varphi:0}, we have that
\begin{equation}\label{eqn:lemma:properties:varphi:2}
\frac{\varphi'(|z|)}{\varphi'(|a|)}
=\frac{|a|}{|z|}
   \frac{\int_0^{|z|}s\Delta\varphi(s)\,ds}{\int_0^{|a|}s\Delta\varphi(s)\,ds}
\qquad(z,a\in\D\setminus\{0\})
\end{equation}
If $z\in D_{\delta_0}(a)$ then  $||z|-|a||<\delta_0\tau(|a|)$ so $|a|-\delta_0\tau(|a|)\le|z|\le|a|+\delta_0\tau(|a|)$. 
Since $\lim_{|a|\to1^{-}}\tau(|a|)=0$, it follows that
\begin{equation}\label{eqn:lemma:properties:varphi:3}
\lim_{|a|\to1^{-}}
\inf\bigl\{\tfrac{|z|}{|a|}:z\in D_{\delta_0}(a)\bigr\}=
\lim_{|a|\to1^{-}}
\sup\bigl\{\tfrac{|z|}{|a|}:z\in D_{\delta_0}(a)\bigr\}=1.
\end{equation}
Moreover, there is $0<r_0<1$ such that $0<|a|-\delta_0\tau(|a|)$,
 for $r_0<|a|<1$, and so there are constants $C>c>0$ such that
\begin{equation}\label{eqn:lemma:properties:varphi:4}
c\,\frac{\int_0^{|a|-\delta_0\tau(|a|)}\frac{s}{\tau(s)^2}\,ds}{\int_0^{|a|}\frac{s}{\tau(s)^2}\,ds}
\le   
\frac{\int_0^{|z|}s\Delta\varphi(s)\,ds}{\int_0^{|a|}s\Delta\varphi(s)\,ds}
\le C\,
\frac{\int_0^{|a|+\delta_0\tau(|a|)}\frac{s}{\tau(s)^2}\,ds}{\int_0^{|a|}\frac{s}{\tau(s)^2}\,ds},
\end{equation}
 for $r_0<|a|<1$ and $z\in D_{\delta_0}(a)$. 
By \ref{item1:properties:varphi}, $\int_0^1\frac{s}{\tau(s)^2}\,ds=\infty$, so we may apply L'H\^{o}pital's rule to get that
\begin{equation}\label{eqn:lemma:properties:varphi:5}
\lim_{r\to1^{-}}\frac{\int_0^{r\pm\delta_0\tau(r)}\frac{s}{\tau(s)^2}\,ds}{\int_0^{r}\frac{s}{\tau(s)^2}\,ds}
=\lim_{r\to1^{-}}\frac{\bigl(1\pm\delta_0\tfrac{\tau(r)}{r}\bigr)\bigl(1\pm\delta_0\tau'(r)\bigr)\,\tau(r)^2}{\tau\bigl(r\pm\delta_0\tau(r)\bigr)^2}=1,
\end{equation}
because $\lim_{r\to1^{-}}\frac{\tau(r)}{\tau(r\pm\delta_0\tau(r))}=1$.
In fact, since ${\displaystyle\lim_{r\to1^{-}}\tau(r)=\lim_{r\to1^{-}}\tau'(r)=0}$, the mean value theorem shows that
\begin{equation*}
|\tau(r\pm\delta_0\tau(r))-\tau(r)|\le c(r)\tau(r), \quad\mbox{with}\quad\lim_{r\to1^{-}}c(r)=0,
\end{equation*}
and so ${\displaystyle\lim_{r\to1^{-}}\tfrac{\tau(r\pm\delta_0\tau(r))}{\tau(r)}=1}$.
Finally, \eqref{eqn:lemma:properties:varphi:1} easily follows 
from~{\eqref{eqn:lemma:properties:varphi:2}-\eqref{eqn:lemma:properties:varphi:5}.}

Note that \eqref{eqn:lemma:properties:varphi:1} implies that there is a radius $0<r<1$ such that 
\begin{equation}\label{eqn:lemma:properties:varphi:6}
1+\varphi'(|z|)\simeq1+\varphi'(|a|)\qquad(r<|a|<1,\,z\in D_{\delta_0}(a))
\end{equation}
If $|a|\le r$ then  
$|\tau(a)-\tau(0)|\le C\,r$, so
\[
\overline{D}(a,\delta\tau(a))\subset\overline{D}(0,r+\delta\tau(a))\subset \overline{D}(0,R_{\delta,r}).
\]
where $R_{\delta,r}=r+\delta(\tau(0)+Cr)>0$. Let
 $\delta=\min\{\delta_0,\,\frac12\frac{1-r}{\tau(0)+Cr}\}$. Then $R_{\delta,r}<1$, and therefore a continuity argument shows that
\begin{equation}\label{eqn:lemma:properties:varphi:7}
1+\varphi'(|z|)\simeq1+\varphi'(|a|)\qquad(|a|\le r,\,z\in D_{\delta}(a)).
\end{equation}
Hence \eqref{eqn:lemma:properties:varphi:6} and \eqref{eqn:lemma:properties:varphi:7} gives the second estimate of \ref{item6:properties:varphi}, and that ends the proof of the proposition.
\end{proof}

The next result applied to $\psi(r)=\varphi(r)+r$ will be used in the proof of Lemma~\ref{le:crecimientog-gprima}.

\begin{proposition}\label{prop:properties:varphi2}
Let $\psi$ be an increasing convex $C^2$ function on $[0,1)$ such that $\psi'(0)>0$,  $\lim_{r\to 1^-}(1-r)\psi'(r)=\infty$ and  $\lim_{r\to 1^-}\frac{\psi''(r)}{\left(\psi'(r) \right)^2}=0$. Then there is $\delta>0$ such that
\begin{align}
\label{eq:compatibility}
&r+\frac{\delta}{\psi'(r)}<1, \quad\text{ for any $r\in [0,1)$,}\qquad\mbox{and}
\\
\label{eq:cond-crecimiento}
&\sup_{0\le r<1}\frac{\psi'\left(r+\frac{\delta}{\psi'(r)}\right)}{\psi'(r)}<\infty.
\end{align}
\end{proposition}
\begin{proof}
Since $\lim_{r\to 1^-}(1-r)\psi'(r)=\infty$, there exists $\delta>0$ satisfying \eqref{eq:compatibility}.
Now let us consider an increasing sequence $\{r_n\}_{n=0}^\infty$ in $[0,1)$ satisfying that $\psi'(r_n)=e^{n}\psi'(0)$, for any $n\ge0$.
It is clear that $\lim_{n\to\infty} r_n=1$, 
 and let us prove that  there is $n_0\in \N$ such that 
\begin{equation}\label{eq:cond-crecimiento-discreta-1}
 r_n+\frac{\delta}{\psi'(r_n)}\le r_{n+1}, \quad \text{for any $n\ge n_0$},
\end{equation}
Indeed, by the main value theorem, for each $n\in\N$ there is $x_n\in (r_n,r_{n+1})$ such that  
$$(e-1)\psi'(r_{n})=\psi'(r_{n+1})-\psi'(r_{n})=\psi''(x_n)(r_{n+1}-r_n),$$
so the inequality in \eqref{eq:cond-crecimiento-discreta-1} is equivalent to
\begin{equation}\label{eq:cond-crecimiento-discreta-10}
\frac{\delta\psi''(x_n)}{(\psi'(r_n))^2}<e-1, \quad \text{for any $n\ge n_0$}.
\end{equation}
By the convexity of $\psi$ 
$$\frac{\delta\psi''(x_n)}{(\psi'(r_n))^2}= \frac{\delta\psi''(x_n)}{(\psi'(x_n))^2}\frac{(\psi'(x_n))^2}{(\psi'(r_n))^2}
\le \frac{\delta\psi''(x_n)}{(\psi'(x_n))^2}\frac{(\psi'(r_{n+1}))^2}{(\psi'(r_n))^2}=e^2 \frac{\delta\psi''(x_n)}{(\psi'(x_n))^2},$$
which together with the hypothesis $\lim_{r\to 1^-}\frac{\psi''(r)}{\left(\psi'(r) \right)^2}=0$ implies that \eqref{eq:cond-crecimiento-discreta-10} holds.
Next observe that \eqref{eq:cond-crecimiento-discreta-1} implies that 
\begin{equation*}
A=\sup_{n\in\N}\frac{\psi'\left(r_n+\frac{\delta}{\psi'(r_n)}\right)}{\psi'(r_n)}<\infty.
\end{equation*}
Now let us consider the function $h(r)=r+\frac{\delta}{\psi'(r)}$. Since $h'(r)=1-\delta\frac{\psi''(r)}{\left(\psi'(r) \right)^2}$ and $\lim_{r\to 1^-}\frac{\psi''(r)}{\left(\psi'(r) \right)^2}=0$, there exists $n_1\in \N$ such that $h$ is increasing in the interval $[r_{n_1},1)$. So
\begin{equation*}
\sup_{r\in [r_{n},r_{n+1})}
\frac{\psi'\left(r+\frac{\delta}{\psi'(r)}\right)}{\psi'(r)}
\le \frac{ \psi'\left( r_{n+1}+\frac{\delta}{\psi'(r_{n+1})}\right) }{\psi'(r_n)}\le e A,
\quad\mbox{for any $n\ge n_1$,}
\end{equation*}
which implies 
$\sup_{ r_{n_1}\le r<1}
 \frac{\psi'\left(r+\frac{\delta}{\psi'(r)}\right)}{\psi'(r)}<\infty.
$ This condition is equivalent to 
\eqref{eq:cond-crecimiento} and this finishes the proof. 
\end{proof}

\subsection{Operator theoretic results}
In the statement of the next result, proved in \cite[Theorem~3.1]{Aleman:Cascante:Fabrega:Pascuas:Pelaez2}, we denote by $\Pi_0:\H(\D)\to\H_0(\D)$ the operator given by $\Pi_0f=f-f(0)=f_0$.

\begin{theorem}
Let $L\in W_g(\ell,m,n)$, where $\ell,m,n\in\N_0$, $m+n\ge1$, and let $k=\ell+m$. Then there exist integers $a_j,b_j$, $j=1,\dots,k$, which do not dependent on $g$ and satisfy
\begin{align}
\label{eqn:global:decomposition}
L&= (1-\delta_L) S^k_gT^n_g+\delta_LS^k_gT_g^n\Pi_0 \\
\nonumber
 &\quad +\sum_{j=1}^ka_j\,S_g^{k-j}T_g^{n+j}
        +\sum_{j=1}^kb_j\,S_g^{k-j}T_g^{n+j}\Pi_0,
\end{align}	
where $\delta_L=0$, if $L$ ends in $T_gM^i_g$, for some $i\in\N_0$, and  $\delta_L=1$, if $L$ ends in $S_gM^i_g$, for some $i\in\N_0$. In particular,
\begin{equation}
\label{eqn:global:decomposition:H0}
L= S^k_gT^n_g+\sum_{j=1}^kc_j\,S_g^{k-j}T_g^{n+j}
\quad\mbox{on $\H_0(\D)$,}	
\end{equation}
where the $c_j$'s are integers independent of $g$.
\end{theorem}

Throughout the rest of the paper we will use the following  two fundamental results. Recall that a $g$-operator is just a linear combination of $g$-words (not necessarily having the same number letters), and $g_r(z)=g(rz)$, for any $g\in\H(\D)$.

\begin{proposition}[{\cite[Proposition 2.4]{CasFabPasPelJFA2024}}]\label{prop:norm:dilations}  
Let $\omega$ be a radial weight, $0<p<\infty$, and let  $L_g$ be a $g$-operator, where $g\in\H(\D)$.
If  $L_g\in\BB(A^p_\omega)$  then $L_{g_r}\in\BB(A^p_\omega)$ and $\|L_{g_r}\|_{A^p_\omega}\lesssim \|L_g\|_{A^p_\omega}$, for any $0<r<1$.
Moreover, if ${\displaystyle\varliminf_{r\nearrow1}\|L_{g_r}\|_{A^p_\omega}<\infty}$,
then $L_g\in\BB(A^p_\omega)$ and 
${\displaystyle\|L_g\|_{A^p_\omega}\simeq\varliminf_{r\nearrow1}\|L_{g_r}\|_{A^p_\omega}}$.
 \end{proposition}

\begin{theorem}[{\cite[Theorems~1.2 and 1.3~a]{CasFabPasPelJFA2024}}]\label{thm:compo}
Let $\omega$ be a radial weight, $g\in \H(\D)$,  and $0<p<\infty$. If $L_g\in W_g(\ell,m,n)$, where $\ell,m,n\in\N_0$
 and $n\in\N$, is bounded on $\Ap2(0)$, then $T_g$ is bounded on $\Ap2$ and  
$\|T_g\|_{\Ap2}\lesssim \opnorm{L_g}_{\Ap2}^{1/(m+n)}$.
\end{theorem} 

It is worth noticing that bearing in mind Theorem~\ref{thm:compo} together with Proposition~\ref{pr:radialized:symbol:Bloch} it is enough to prove Theorem~\ref{thm:sharp:estimate:norm:word} for $g\in\H(\overline{\D})$.

\subsection{Estimates of the Bergman reproducing kernel for $A^2_{\omega}$}\label{subsec:Bergman:kernels}
$\mbox{}$\newline
In this section we recall  some known estimates  of the Bergman reproducing kernel for $A^2_{\omega}$.
Let $\omega\in\SW$. Then the functions in $\Ap2$, $0<p<\infty$, satisfy the estimate  
\begin{equation}\label{eqn:growth:of:Apomegap:functions}
|f(z)|\lesssim \tau(z)^{-\frac{2}{p}}\om(z)^{-\frac12}\|f\|_{\Ap2} \qquad(f\in\Ap2,\,z\in\D)
\end{equation}
(see~{\cite[Lemma~2.2]{Pau-Pelaez:JFA2010}}), which shows that the point evaluation functionals on $\Ap2$ are bounded. As a consequence, since $A^2_{\omega}$ is a Hilbert space, for any $a\in\D$ there is a unique function $K_a^\omega\in A^2_{\omega}$ such that
\[
f(a)=\int_{\D}f(w)\,\overline{K_a^\omega(w)}\,\omega(w)\,dA(w)
\qquad(f\in A^2_{\omega}).
\]
The function $K_a^\omega$ is called the {\em Bergman reproducing kernel for $A^2_{\omega}$} at $a$. Then it is well known that
\begin{equation}\label{eqn:power:series:Bergman:kernel}
K_a^\omega(w)=\sum_j\frac{\overline{a}^j}{\alpha_j}\,w^j\qquad(w\in\D),
\end{equation}
where the convergence of the series is in $A^2_{\omega}$ and 
$\alpha_j=2\int_0^1r^{2j+1}\omega(r)\,dr$.
Since $\alpha_j\ge\int_r^1s^{2j+1}\omega(s)\,ds\ge r^{2j+1}\int_r^1\omega(s)\,ds$, for any $0<r<1$, we have that
\[
\limsup_{j\to\infty}\alpha_j^{-1/j}\le\inf_{0<r<1}1/r^2=1,
\]
and therefore the radius of convergence $R_a$ of the power series in \eqref{eqn:power:series:Bergman:kernel} satisfies
\[
R_a=\frac1{|a|}\biggl(\limsup_{j\to\infty}\alpha_j^{-1/j}\biggr)^{-1}\ge\frac1{|a|}>1.
\]
 Hence $K_a^{\om}\in\H(\overline{\D})$, for every $a\in\D$.
\medskip

Since $\SW$ is included in the class of weights studied in~{\cite{HLS}}, the estimates of the Bergman kernel obtained there apply to $K^{\omega}_a$, for $\omega\in\SW$. Namely, 
 there is $\eta>0$ such that $K_a^\omega$ satisfies the global upper estimate 
\begin{equation} \label{5.3} 
|K_a^\omega(z)| 
\lesssim \frac{\omega(a)^{-\frac12}\omega(z)^{-\frac12}}{\tau(a)\tau(z)} \,e^{-\eta d_\tau(a,z)}\qquad (a,z\in\D),  
\end{equation}
and there is $\delta>0$ such that $K_a^\omega$ satisfies the local lower estimate 
\begin{equation} \label{5.3bis}
|K_a^\omega(z)|\gtrsim
\frac{\omega(a)^{-\frac12}\omega(z)^{-\frac12}}{\tau(a)\tau(z)}
\qquad(a,z\in\D,\,d_\tau(a,z)\le\delta).  
\end{equation}
(see \cite[Theorem 3.2]{HLS}).
In addition, the  factor $e^{-\eta d_\tau(a,z) }$ has, for any $M >0$, the upper estimate 
\begin{equation}\label{5.4} 
e^{-\eta d_\tau(a,z) }\lesssim
\Big(\frac{\min(\tau(a),\tau(z) )}{|a-z|}\Big)^M
\qquad(a,z\in\D)
\end{equation}
(see \cite[(23)]{HLS}).
Here $d_\tau$ is the distance defined by 
\begin{equation*} 
d_\tau(a,z) = \inf_\gamma \int_0^1\frac{|\gamma'(t)|}{\tau(\gamma(t))}\,dt\qquad(a,z\in\D),
\end{equation*}
where the infimum is taken over all piecewise $C^1$ curves 
$\gamma:[0,1] \to \D$ with $\gamma(0)=a$ and $\gamma(1)=z$.
Recall that $d_{\tau}$ is comparable to the distance on $\D$ induced by the Bergman metric $\frac12\frac{\partial^2\,}{\partial z\partial\overline{z}}\log K^{\omega}_z(z)\,dz\otimes d\overline{z}$ 
(see \cite[page 355]{Constantin:Ortega-Cerda}).

Note that if $\delta$ is as in 
Proposition~{\ref{prop:properties:varphi} \ref{item6:properties:varphi}} we have 
\begin{equation*}
d_{\tau}(a,z)\le|z-a|\int_0^1\frac{dt}{\tau\bigl(a+t(z-a)\bigr)}\simeq\frac{|z-a|}{\tau(a)}\qquad(a\in\D,\,z\in D_{\delta}(a)),
\end{equation*}
and, in particular, $D_{a,\delta}\subset\{z\in\D:d_{\tau}(z,a)<C\delta\}$, for $\delta>0$ small enough and for an absolute constant $C>0$. 
As a consequence, \eqref{5.3} and \eqref{5.3bis} give 
\begin{equation} \label{5.3tris}
|K_a^\omega(\zeta)|\simeq
\frac{\omega(a)^{-\frac12}\omega(z)^{-\frac12}}{\tau(a)\tau(\zeta)}\qquad(a\in\D,\,\zeta\in D_{\delta}(a)),  
\end{equation}
for $\delta>0$ small enough. Finally, we have the following weighted $L^{\infty}$ and $A^p_{\omega_p}$ norm estimates of the Bergman kernel
(see \cite[Corollary~3.2]{HLS}):
\begin{align}
\label{eqn:estimate:weighted:Linfty:norm:Bergman:kernel}
\sup_{z\in\D}|K_a^{\omega}(z)|\om(z)^{\frac12}
&\simeq\om(a)^{-\frac12}\tau^{-2}(a)\qquad(a\in\D).
\\
\label{eqn:estimate:Apomegap:norm:Bergman:kernel}
\|K_a^{\omega}\|_{A^p_{\omega_p}}
&\simeq \om(a)^{-\frac12}\tau(a)^{\frac2p-2}\qquad(a\in\D).
\end{align}

\section{$\varphi$-Bloch classes of power functions}\label{section:symbols}

In this section we will study our $\varphi$-Bloch classes of power functions 
$\B^q_{\varphi}$.
 From now on we will assume that $\varphi$ satisfies the three conditions \ref{rapidly:decreasing:condition:a}-\ref{rapidly:decreasing:condition:c}.

The first result will allow us to reduce the estimate of  $\|g\|_{\B^q_\varphi}$ to the one of $\|g_r\|_{\B^q_\varphi}$, where $g_r(z):=g(rz)$, for $0<r<1$.
\begin{proposition}\label{pr:radialized:symbol:Bloch}
\begin{equation}\label{eqn:radialized:symbol:Bloch}
\sup_{0<r<1}\|g_r\|_{\B^q_\varphi}=
\|g\|_{\B^q_\varphi}
=\lim_{r\to1^{-}}\|g_r\|_{\B^q_\varphi}
\qquad(g\in\H(\D),\,q\ge1).
\end{equation}
\end{proposition}

\begin{proof}
Let $g\in\H(\D)$ and $q\ge1$. Then
\begin{equation}\label{eqn1:radialized:symbol:Bloch}
\bigl|\nabla|g_r|^q\bigr|(z)=r\bigl|\nabla|g|^q\bigr|(rz)
\qquad(z\in\D,\,0<r<1),
\end{equation}
so 
$\bigl|\nabla|g_r|^q\bigr|(z)
\le\|g\|^q_{\B^q_{\varphi}}\bigl(1+\varphi'(r|z|)\bigr)
\le\|g\|^q_{\B^q_{\varphi}}\bigl(1+\varphi'(|z|)\bigr)$,
and therefore 
${\displaystyle
\sup_{0<r<1}\|g_r\|_{\B^q_\varphi}\le\|g\|_{\B^q_\varphi}}$.
Moreover, ${\displaystyle \|g\|_{\B^q_\varphi}\le\liminf_{r\to1^{-}}\|g_r\|_{\B^q_\varphi}}$, since
the continuity of $\bigl|\nabla|g|^q\bigr|$ on $\D$ and  \eqref{eqn1:radialized:symbol:Bloch} imply that 
\begin{align*}
\bigl|\nabla|g|^q\bigr|(z)
&=\lim_{r\to1^{-}}r\bigl|\nabla|g|^q\bigr|(rz)
   =\lim_{r\to1^{-}}\bigl|\nabla|g_r|^q\bigr|(z)
\\
&\le \liminf_{r\to1^{-}}\biggl(\|g_r\|_{\B^q_\varphi}^q\bigl(1+\varphi'(r|z|)\bigr)\biggr)
\\
&\le \biggl(\liminf_{r\to1^{-}}\|g_r\|_{\B^q_\varphi}^q\biggr)
       \bigl(1+\varphi'(|z|)\bigr),\quad\mbox{for every $z\in\D$,}
\end{align*}
where the last inequality holds because $\varphi'$ is non-decreasing.
Finally, the chain of inequalities 
\[
\liminf_{r\to1^{-}}\|g_r\|_{\B^q_\varphi}
\le\limsup_{r\to1^{-}}\|g_r\|_{\B^q_\varphi}
\le\sup_{0<r<1}\|g_r\|_{\B^q_\varphi}
\]
ends the proof.
\end{proof}

Next we will give a Lipschitz-type description of $\B^q_{\varphi}$ with respect to the distance $\beta_{\varphi}$ on $\D$ induced by the Riemannian metric
$\frac12\bigl(1+\varphi'(|z|)\bigr)(dz\otimes d\overline{z})$, that is,
\begin{equation}
\beta_{\varphi}(z,w):=
\inf_{\gamma\in\Gamma(z,w)}
\int_0^1\bigl(1+\varphi'(|\gamma(t)|)\bigr)\,|\gamma'(t)|\,dt
\qquad(z,w\in\D),
\end{equation}
where $\Gamma(z,w)$ is the set of all piecewice $C^1$ curves $\gamma:[0,1]\to\D$ satisfying that $\gamma(0)=z$ and $\gamma(1)=w$.
For example, if $\varphi_{\alpha}(z)=\frac{\alpha}2\log\frac1{1-|z|^2}$,
$\alpha>0$, then $\beta_{\varphi_{\alpha}}$ is comparable to the the hyperbolic distance on $\D$, {\em i.e.}
\[
\beta_{\varphi_{\alpha}}(z,w)\simeq
\log\frac{1+|\frac{z-w}{1-\overline{w}z}|}{1-|\frac{z-w}{1-\overline{w}z}|}
\qquad(z,w\in\D).
\]
Note that
\begin{equation}\label{eqn:growth:betaphiz0}
\beta_{\varphi}(z,0)\le\bigl(1+\tfrac1{\varphi(0)}\bigr)\,\varphi(|z|)
\qquad(z\in\D),
\end{equation}
because
\begin{align*}
\beta_{\varphi}(z,0)
&\le\int_0^1\bigl(1+\varphi'(t|z|)\bigr)\,|z|\,dt
  =\int_0^{|z|}\bigl(1+\varphi'(r)\bigr)\,dr
\\
&\le 1+\varphi(|z|)\le\bigl(1+\tfrac1{\varphi(0)}\bigr)\,\varphi(|z|),
\quad\mbox{for any $z\in\D$.} 
\end{align*}

The following proposition estimates the size of $\beta_{\varphi}(z,w)$, for $z\in\D$ and $w\in\D$ in a small neighborhood of $z$.

\begin{proposition}
Let $\delta>0$ be as in Proposition~{\textup{\ref{prop:properties:varphi}}\,\,\ref{item6:properties:varphi}}.
Then 
\begin{equation}\label{eqn:behavior:betaphi:close:points}
\beta_{\varphi}(z,w)\simeq|z-w|\bigl(1+\varphi'(|z|)\bigr)
\qquad(z\in\D,\,w\in\overline{D}_{\delta}(z)).
\end{equation}
\end{proposition}

\begin{proof}
If $z,w\in\D$ then
\begin{align*}
\beta_{\varphi}(z,w)
&\le |z-w|\int_0^1\bigl(1+\varphi'(|z+t(w-z)|)\bigr)\,dt
\\
&\le |z-w|\int_0^1\bigl(1+\varphi'(|z|+t|w-z|)\bigr)\,dt, 
\end{align*}
so
\begin{equation}\label{eqn1:behavior:betaphi:close:points}
\beta_{\varphi}(z,w)\lesssim|z-w|\bigl(1+\varphi'(|z|)\bigr)
\qquad(z\in\D,\,w\in\overline{D}_{\delta}(z)).
\end{equation}
On the other hand, let  $\gamma\in\Gamma(z,w)$, with $z\in\D$ and
 $w\in\overline{D}_{\delta}(z)$. 
When $\gamma([0,1])\subset\overline{D}_{\delta}(z)$, we have that 
\begin{equation*}
\int_0^1\bigl(1+\varphi'(|\gamma(t)|)\bigr)|\gamma'(t)|\,dt
\simeq
\bigl(1+\varphi'(|z|)\bigr)\int_0^1|\gamma'(t)|\,dt
\ge|z-w|(1+\varphi'(|z|)\bigr).
\end{equation*}
If $\gamma([0,1])\not\subset\overline{D}_{\delta}(z)$, then 
$t_0=\inf\{t\in[0,1]:\gamma(t)\in\partial D_{\delta}(z)\} $
satisfy that $0<t_0\le1$, $\gamma(t_0)\in\partial D_{\delta}(z)$, and $\gamma(t)\in\overline{D}_{\delta}(z)$, for $t\in[0,t_0]$,   so 
\begin{align*}
\int_0^1\bigl(1+\varphi'(|\gamma(t)|)\bigr)|\gamma'(t)|\,dt
&\gtrsim
\bigl(1+\varphi'(|z|)\bigr)\int_0^{t_0}|\gamma'(t)|\,dt
\\
&\ge
\bigl(1+\varphi'(|z|)\bigr)\,|\gamma(t_0)-z|
=\delta\tau(z)\bigl(1+\varphi'(|z|)\bigr)
\\
&\ge|z-w|\bigl(1+\varphi'(|z|)\bigr).
\end{align*}
Therefore
\begin{equation}\label{eqn2:behavior:betaphi:close:points}
\beta_{\varphi}(z,w)\gtrsim|z-w|\bigl(1+\varphi'(|z|)\bigr)
\qquad(z\in\D,\,w\in\overline{D}_{\delta}(z)).
\end{equation}
Hence \eqref{eqn:behavior:betaphi:close:points} directly follows from \eqref{eqn1:behavior:betaphi:close:points} and \eqref{eqn2:behavior:betaphi:close:points}.
\end{proof}

\begin{proposition}
\begin{equation}\label{eqn:Bqphi:Lipschitz}
\|g\|_{\B^q_\varphi}^q
=\sup_{\substack{z,w\in\D\\ z\ne w}}
\frac{\bigl||g(z)|^q-|g(w)|^q\bigr|}{\beta_{\varphi}(z,w)}
\qquad(g\in\H(\D),q\ge1).
\end{equation}
\end{proposition}

\begin{proof} Let $M_{g,\varphi}$ be the supremum at the statement, and let
 $\gamma\in\Gamma(z,w)$, where $z,w\in\D$.
Then, since $|g|^q\in C^1(\D)$ (because $q\ge1$), we have  that
\begin{align*}
\bigl||g(z)|^q-|g(w)|^q\bigr|
&\le\int_0^1|\nabla|g|^q|(\gamma(t))\,|\gamma'(t)|\,dt
\\
&\le
\|g\|_{\B^q_\varphi}^q
\int_0^1(1+\varphi'(|\gamma(t)|))\,|\gamma'(t)|\,dt.
\end{align*}
Therefore $M_{g,\varphi}\le\|g\|_{\B^q_\varphi}^q$. 
On the other hand,  
\begin{equation*}
|\nabla|g|^q|(z)
=\limsup_{w\to z}
\frac{||g(w)|^q-|g(z)|^q|}{|w-z|}
\le M_{g,\varphi}
\limsup_{w\to z}
\frac{\beta_{\varphi}(w,z)}{|w-z|}
\quad(z\in\D).
\end{equation*}
But
\begin{equation*}
\limsup_{w\to z}\frac{\beta_{\varphi}(w,z)}{|w-z|}
\le\lim_{w\to z}\int_0^1\bigl(1+\varphi'(|z+t(w-z)|)\bigr)\,dt
=1+\varphi'(|z|),
\end{equation*}
and hence we obtain that $\|g\|_{\B^q_\varphi}^q\le M_{g,\varphi}$, which completes the proof.
\end{proof}

Next we prove the radicality estimate \eqref{eq:i1} which will be a key tool to prove
our main result.

\begin{proof}[{\bf Proof of Theorem \ref{prop:radicality:estimateintro}}]

Assume that $\|g\|_{\B^{q_2}_\varphi}=1$. Then
\begin{equation*}
|\nabla|g|^{q_1}|(z)=\tfrac{q_1}{q_2}|g(z)|^{q_1-q_2}|\nabla|g|^{q_2}|(z)
\le\tfrac{q_1}{q_2}|g(z)|^{q_1-q_2}\bigl(1+\varphi'(|z|)\bigr),
\end{equation*}
so $|\nabla|g|^{q_1}|(z)\le\tfrac{q_1}{q_2}\bigl(1+\varphi'(|z|)\bigr)$, 
whenever $|g(z)|\ge1$.\vspace*{3pt}

Now let $z\in\D$ such that  $|g(z)|<1$. We want to estimate $|g'(z)|$ from above. Let $\delta>0$ as in Proposition~{\ref{prop:properties:varphi}\,\,\ref{item6:properties:varphi}}.
 Recall that, by Cauchy estimates, 
\[
|g'(z)|\le\tfrac1{\delta\tau(z)}\sup_{w\in\partial D_{\delta}(z)}|g(w)|
=\tfrac1{\delta\tau(z)}\,M_{\delta}(z).
\]
But \eqref{eqn:Bqphi:Lipschitz} shows that 
\begin{equation*}
|g(w)|^{q_2}
\le \bigl||g(w)|^{q_2}-|g(z)|^{q_2}\bigr|+|g(z)|^{q_2}
\le\beta_{\varphi}(w,z)+1\qquad(w\in\D),
\end{equation*}
while \eqref{eqn:behavior:betaphi:close:points} gives that
\begin{equation*}
\sup_{w\in\partial D_{\delta}(z)}\beta_{\varphi}(w,z)
\simeq\tau(z)\bigl(1+\varphi'(|z|)\bigr).
\end{equation*}
It follows that $M_{\delta}(z)\lesssim \tau(|z|)^{\frac1{q_2}}\bigl(1+\varphi'(|z|)\bigr)^{\frac1{q_2}}$
 and therefore
\[
|g'(z)|
\le\tfrac1{\delta\tau(z)}M_{\delta}(z)
\lesssim\tau(|z|)^{\frac1{q_2}-1}\bigl(1+\varphi'(|z|)\bigr)^{\frac1{q_2}}
\lesssim1+\varphi'(|z|),
\]
since $q_2>1$ and $\tau(|z|)\bigl(1+\varphi'(|z|)\bigr)\to\infty$, as $|z|\to1^{-}$.
 Hence we conclude that $|\nabla|g|^{q_1}|(z)\lesssim1+\varphi'(|z|)$, and that ends the proof.
\end{proof}

We end this section with a growth description of $\B^q_{\varphi}$ when $\varphi$ also satisfies
\begin{equation}\label{varphi}
		\sup_{0\le r<1}\frac{\varphi''(r)\varphi(r)}{(1+\varphi'(r))^2}<\infty.
	\end{equation}
Namely, we will show that if $\varphi$ satisfies the extra condition \eqref{varphi} then $\B^q_{\varphi}$ coincides with the growth class $H^{\infty,q}_\varphi$ defined by \eqref{eqn:def:Hintyqvarphi}.

\begin{proposition}\label{prop:Bqphi:as:a:growth:space}
Assume that $\varphi$ satisfies \eqref{varphi}
\textup{(}besides conditions \ref{rapidly:decreasing:condition:a}-\ref{rapidly:decreasing:condition:c}\textup{)}, and let $\psi(r)=r+\varphi(r)$. Then, for any $q\ge1$, we have the identities
\begin{equation}\label{eqn:Bqphi:as:a:growth:space}
H^{\infty,q}_{\varphi}
=\{g\in\H(\D):g'\in H^{\infty,q}_{\psi^{1-q}(\psi')^q}\}
=\B^q_{\varphi},
\end{equation}
\end{proposition}

\begin{remark}
It is worth mentioning that, even in the case when $\varphi$ is an increasing function such that $\lim_{r\to 1^-}\varphi(r)=\infty$, it is necessary to assume some extra condition of the type \eqref{varphi} in order to ensure that, for any $g\in\H(\D)$,
 $M_\infty(r,g)=\sup_{|z|=r}|g(z)|= O(\varphi(r))$, as $r\to 1^{-}$, and
$M_\infty(r,g')= O(\varphi'(r))$, as $r\to 1^{-}$,
are equivalent conditions. Indeed, for
$\varphi(r)=\left( \log\frac{e}{1-r}\right)^2$ there exists $g\in \H(\D)$ such that $M_\infty(r,g)= O(\varphi(r))$, as $r\to 1^{-}$, but
$\limsup_{r\to 1^-} \frac{M_\infty(r,g')}{\varphi'(r)}=\infty$.
In fact, take $g(z)=h(z)\log\frac{e}{1-z}$, where $h(z)=\sum_{k=0}^\infty 2^k z^{2^{2^k}}$ and $\log$ denotes the principal branch of the logarithm. By an straightforward calculation $M_\infty(r,h)= O(\log\frac{e}{1-r})$, so 
\[
M_\infty(r,g)= O\left(\left(\log\frac{e}{1-r}\right)^2\right).
\]
Since the sequence of  Taylor coefficients of $h$ is unbounded, $h$ is not a Bloch function, so 
$\limsup_{r\to1^-}(1-r)M_\infty(r,h')=\infty$. 
On the other hand, $M_\infty(r,h')=h'(r)$ because the Taylor coefficients of $h'$ are nonnegative. 
Therefore $\limsup_{r\to1^-}(1-r)h'(r)=\infty$.
 Moreover,
\begin{equation*}
g'(z)=h'(z)\log\frac{e}{1-z}+\frac{h(z)}{1-z},
\end{equation*}
and so 
$$M_{\infty}(r,g')\ge g'(r)=h'(r)\log\frac{e}{1-r}+\frac{h(r)}{1-r}\ge h'(r)\log\frac{e}{1-r}, \quad 0\le r<1.$$
Consequently,
\begin{equation*}
\limsup_{r\to1^-}
\frac{M_{\infty}(r,g')}{\frac1{1-r}\log\frac{e}{1-r}}
\ge
\limsup_{r\to1^-}(1-r)h'(r)=\infty,
\end{equation*}
which implies that 
$\limsup_{r\to1^-}\frac{M_\infty(r,g')}{\varphi'(r)}=\infty$.
\end{remark}

The main tool to prove Proposition~{\ref{prop:Bqphi:as:a:growth:space}} is the following result which is a quantitative version of~{\cite[Theorem~D]{PavP}} (see also \cite[Theorem 2.1 and (ii) of p. 740]{Pa1}). We include a proof for the sake of completeness.

\begin{theorem}\label{pavl}
Let  $\psi\in C^2[0,1)$ such that $\psi$ and $\psi'$ are positive on $[0,1)$
and $\lim_{r\to1^-}\psi(r)=\infty$. Assume that $\psi$ satisfies the condition 
\begin{equation}\label{psi}
\sup_{0\le r<1}\frac{\psi''(r)\psi(r)}{(\psi'(r))^2}<\infty.
\end{equation}
Then
\begin{equation}\label{eqnorms}
\sup_{0\le r<1} \frac{M_\infty(r,g)}{\psi(r)}\simeq |g(0)|+ \sup_{0\le r<1} \frac{M_\infty(r,g')}{\psi'(r)}\qquad(g\in\H(\D)).
\end{equation}
\end{theorem}

\begin{proof}
The estimate
\[
\sup_{0\le r<1} \frac{M_\infty(r,g)}{\psi(r)}\lesssim |g(0)|+ \sup_{0\le r<1} \frac{M_\infty(r,g')}{\psi'(r)}
\qquad(g\in\H(\D))
\]
can be easily proved as follows:
\begin{align*}
|g(z)|
&\le|g(0)|+\int_0^z|g'(\zeta)||d\zeta|
\le|g(0)|\frac{\varphi(|z|)}{\varphi(0)}
+\int_0^{|z|}M_{\infty}(r,g')dr\\
&\le|g(0)|\frac{\psi(|z|)}{\psi(0)}
+\biggl(\sup_{0\le r<1} \frac{M_\infty(r,g')}{\psi'(r)}\biggr)\int_0^{|z|}\psi'(r)dr\\
&\le|g(0)|\frac{\psi(|z|)}{\psi(0)}
+\psi(|z|)\sup_{0\le r<1} \frac{M_\infty(r,g')}{\psi'(r)}
\\
&\le\max\biggl(1,\frac1{\psi(0)}\biggr)\biggl\{
|g(0)|+\sup_{0\le r<1} \frac{M_\infty(r,g')}{\psi'(r)}
\biggr\}\psi(|z|).
\end{align*}
In order to show the opposite estimate, let us assume without loss of generality that $\psi(0)=1$. Let  $\{r_n\}_{n=0}^\infty$ be the increasing sequence in $[0,1)$ defined by $\psi(r_n)=e^n$ and let 
$M:=\sup_{0\le r<1}\frac{\psi''(r)\psi(r)}{\psi'(r)^2}.$ It follows that
\[
\frac{\psi''(r)}{\psi'(r)}\le M\, \frac{\psi'(r)}{\psi(r)}\qquad(0\le r<1).
\]
By integrating this inequality and taking exponentials we get that, for any $n\in\N_0$, 
\begin{equation}\label{eq:xy}
\frac{\psi'(y)}{\psi'(x)}\le \left( \frac{\psi(y)}{\psi(x)}\right)^M\le  \left( \frac{\psi(r_{n+1})}{\psi(r_n)}\right)^M=e^M 
\quad(r_n\le x\le y\le r_{n+1}).
\end{equation}
On the other hand, by the mean value theorem, there exists 
$x_n\in(r_n, r_{n+1})$ such that
$\psi(r_{n+1})\,\frac{e-1}e=\psi(r_{n+1})-\psi(r_n)=\psi'(x_n) (r_{n+1}-r_n)$.
Applying this identity together with  the well-known inequality
\begin{equation}\label{eq:Cauchy}
M_\infty(r,g')\le C\,\frac{M_\infty(\rho,g)}{\rho-r} 
\qquad(0\le r<\rho<1,\,g\in\H(\D)),
\end{equation}
where $C>0$ is an absolute constant, we get that
\begin{equation*}
M_\infty(r_n,g') 
\le C\,\frac{M_\infty(r_{n+1},g)}{r_{n+1}-r_n}
=\frac{Ce}{e-1}\,\psi'(x_n)\,\frac{M_\infty(r_{n+1},g)}{\psi(r_{n+1})}.
\end{equation*}
Now, since $r_n<x_n<r_{n+1}$, \eqref{eq:xy} gives that
$\psi'(x_n)\le e^M\psi'(r_n)$, so 
\begin{equation}\label{eqn:discrete:inequality:1}
\frac{M_\infty(r_n,g')}{\psi'(r_n)}
\le \frac{Ce^{M+1}}{e-1}\,\frac{M_\infty(r_{n+1},g)}{\psi(r_{n+1})}
\qquad(n\in\N_0).
\end{equation}
Finally, \eqref{eq:xy},  \eqref{eqn:discrete:inequality:1}, and the maximum modulus principle show that
\[
\sup_{r_{n-1}\le r<r_n}\frac{M_\infty(r,g')}{\psi'(r)}
\le e^M\,\frac{M_\infty(r_n,g')}{\psi'(r_n)}
\le\frac{Ce^{2M+1}}{e-1}\,\sup_{0\le r <1}\frac{M_\infty(r,g)}{\psi(r)}
\quad(n\in\N),
\]
which clearly implies that
\[
|g(0)|+ \sup_{0\le r<1} \frac{M_\infty(r,g')}{\psi'(r)}\lesssim \sup_{0\le r<1} \frac{M_\infty(r,g)}{\psi(r)},
\]
since $\cup_{n=1}^{\infty}[r_{n-1},r_n)=[0,1)$.
Hence the proof is complete. 
\end{proof}

\begin{corollary}\label{cor:eqnorms:alpha}
Let $\psi$ be a function as in Theorem~{\ref{pavl}} and let $\psi_{\alpha}:=\psi^\alpha$ on $[0,1)$, for $\alpha>0$.
Then
\begin{equation}\label{eqnorms:alpha}
\sup_{0\le r<1} \frac{M_\infty(r,g)}{\psi_{\alpha}(r)}\simeq |g(0)|+ \sup_{0\le r<1} \frac{M_\infty(r,g')}{\psi_{\alpha}'(r)}\qquad(g\in\H(\D)).
\end{equation}
\end{corollary}

\begin{proof}
It is clear that $\psi_{\alpha}\in C^2[0,1)$, $\lim_{r\to1^-}\psi_{\alpha}(r)=\infty$, and both $\psi_{\alpha}$ and $\psi'_{\alpha}$ are positive on $[0,1)$. Moreover, $\psi_{\alpha}$ satisfies \eqref{psi} since
\[
\frac{\psi_\alpha''(r)\psi_\alpha(r)}{\psi_\alpha'(r)^2}=1-\frac{1}{\alpha}+ \frac{1}{\alpha}\frac{\psi''(r)\psi(r)}{\psi'(r)^2}
\qquad(0\le r<1).
\]
Therefore \eqref{eqnorms:alpha} directly follows from Theorem \ref{pavl}.
\end{proof}

\begin{proof}[{\bf Proof of Proposition~{\ref{prop:Bqphi:as:a:growth:space}}}]
By Proposition~{\ref{prop:properties:varphi} ~\ref{item4:properties:varphi},} 
 $\lim_{r\to 1^-}\varphi(r)=\infty$  so $\psi(r)\simeq\varphi(r)$.
It follows that
$H^{\infty,q}_{\varphi}=H^{\infty,q}_\psi$.  Moreover,    \eqref{varphi} implies \eqref{psi}. Therefore we may apply 
Corollary \ref{cor:eqnorms:alpha} with $\alpha=\frac1q$ to obtain 
 the first identity in \eqref{eqn:Bqphi:as:a:growth:space}.

Finally, let us show that $\B^q_{\varphi}=H^{\infty,q}_{\varphi}$.
If $g\in\B^q_{\varphi}$ then \eqref{eqn:Bqphi:Lipschitz} and  \eqref{eqn:growth:betaphiz0} show that
\[
|g(z)|^q
\le|g(0)|^q+\|g\|^q_{\B^q_{\varphi}}\,\beta_{\varphi}(z,0)
\le\left\{\tfrac{|g(0)|^q}{\varphi(0)}
 +\|g\|^q_{\B^q_{\varphi}}\,\bigl(1+\tfrac1{\varphi(0)}\bigr)\right\}\varphi(|z|),
\]
 so $g\in H^{\infty,q}_{\varphi}$. 
 Conversely, if $g\in H^{\infty,q}_{\varphi}$ then 
$g'\in H^{\infty,q}_{\psi^{1-q}(\psi')^q}$, so 
\begin{equation*}
\bigl|\nabla|g|^q\bigr|(z)
\lesssim|g(z)|^{q-1}|g'(z)|
\lesssim\varphi(|z|)^{1-\frac1q}\psi(|z|)^{\frac1q-1}\psi'(|z|)
\le 1+\varphi'(|z|),
\end{equation*}
which means that  $g\in\B^q_{\varphi}$, and that ends the proof.
\end{proof}

\section{Proof of Theorem~{\ref{thm:sharp:estimate:norm:word}~\ref{thm:sharp:estimate:norm:word:b}}: Reduction to the case 
	$L_g=S^m_gT^n_g$}
\label{section:reduction:to:SmTn}

In this section we will deduce Theorem~{\ref{thm:sharp:estimate:norm:word}~\ref{thm:sharp:estimate:norm:word:b}} from the following two key results, which correspond to the case $L_g=S^m_gT^n_g$ and whose proofs will be postponed to Sections~\ref{section:proof:of:proposition4.1} and \ref{section:proof:of:proposition4.2}.

\begin{theorem}\label{prop:sharp:upper:estimate:norm:word:ST}
Let $m\in\N_0$, $n\in\N$, and $s=\frac{m}n+1$. If $\omega=e^{-2\varphi} \in \SW$,  then
\begin{equation}
\label{eqn:sharp:upper:estimate:norm:word:ST}
\|S^m_gT^n_g\|_{\Ap2}
\lesssim\|g\|^{sn}_{\B^s_{\varphi}}
\qquad(g\in\H(\D)).
\end{equation}
\end{theorem}

\begin{theorem}\label{prop:sharp:lower:estimate:norm:word:ST}
Let $m\in\N_0$, $n\in\N$, and $s=\frac{m}n+1$. If $\omega=e^{-2\varphi} \in \SW$,  then
\begin{equation}
\label{eqn:sharp:lower:estimate:norm:word:ST}
\|g\|^{sn}_{\B^s_{\varphi}}
\lesssim
\opnorm{S^m_gT^n_g}_{\Ap2}
\qquad(g\in\H(\D)).
\end{equation}
\end{theorem}

 Let $\ell,m,n\in\N_0$
 such that $n\ge 1$, and let $k=\ell+m$ and $N=k+n$.
We begin by proving that
\begin{equation}
\label{eqn:sharp:upper:estimate:norm:word}
\|L_g\|_{\Ap2}
\lesssim\|g\|^N_{\B^s_{\varphi}}
\qquad(g\in\H(\D),\,L_g\in W_g(\ell,m,n)),
\end{equation}
where $s=\frac{\ell+m}n+1$. 
In fact, taking into account that  
any $L_g\in W_g(\ell,m,n)$
 satisfies~{\eqref{eqn:global:decomposition}, estimates
 \eqref{eqn:sharp:upper:estimate:norm:word:ST} and Theorem~{\ref{prop:radicality:estimateintro}} give  \eqref{eqn:sharp:upper:estimate:norm:word}:
\begin{equation*}
\|L_g\|_{\Ap2}
\lesssim
\|S^k_gT^n_g\|_{\Ap2}
+\sum_{j=1}^k \|S^{k-j}_gT^{n+j}_g\|_{\Ap2}
\lesssim
\sum_{j=0}^k \|g\|^N_{\B^{\frac{k+n}{n+j}}_{\varphi}}
\lesssim
\|g\|^N_{\B^s_{\varphi}}.
\end{equation*}

Now we want to prove that 
\begin{equation}
\label{eqn:sharp:lower:estimate:norm:word}
\|g\|^N_{\B^s_{\varphi}}
\lesssim
\opnorm{L_g}_{\Ap2}
\qquad(g\in\H(\D),\,L_g\in W_g(\ell,m,n)).
\end{equation}
In order to do that, we may assume that $g\in\H(\overline{\D})$, by Proposition~{\ref{prop:norm:dilations}} and~{\eqref{eqn:radialized:symbol:Bloch}}. Assume that $L_g\in W_g(\ell,m,n)$ is bounded on $\Ap2(0)$.
Then taking into account that any $L_g$ satisfies \eqref{eqn:global:decomposition:H0}, the estimates  \eqref{eqn:sharp:upper:estimate:norm:word:ST} and \eqref{eqn:sharp:lower:estimate:norm:word:ST}
together with Theorem~\ref{thm:compo}
and Theorem~{\ref{prop:radicality:estimateintro}} show that
\begin{align*}
\|g\|^N_{\B^s_{\varphi}}
&\lesssim
\opnorm{S^k_gT^n_g}_{\Ap2}
\lesssim
\opnorm{L_g}_{\Ap2}
+\sum_{j=1}^k \|S^{k-j}_gT^{n+j}_g\|_{\Ap2}\\
&\le
\opnorm{L_g}_{\Ap2}
+\|T_g\|_{\Ap2}\sum_{j=1}^k \|S^{k-j}_gT^{n+j-1}_g\|_{\Ap2}\\
&\lesssim
\opnorm{L_g}_{\Ap2}
+\opnorm{L_g}_{\Ap2}^{\frac1N}\sum_{j=1}^m \|g\|^{N-1}_{\B^{\frac{N-1}{n+j-1}}_{\varphi}}\\
&\lesssim
\opnorm{L_g}_{\Ap2}
+\opnorm{L_g}_{\Ap2}^{\frac1N}\,\|g\|^{N-1}_{\B^s_{\varphi}},
\end{align*}
It turns out that either $\opnorm{L_g}_{\Ap2}=\|g\|_{\B^s_{\varphi}}=0$ 
or $0<\opnorm{L_g}_{\Ap2}<\infty$ and 
\begin{equation*}
\frac{\|g\|^N_{\B^s_{\varphi}}}{\opnorm{L_g}_{\Ap2}}
\lesssim 
1+\Biggl(\frac{\|g\|^N_{\B^s_{\varphi}}}{\opnorm{L_g}_{\Ap2}}\Biggr)^{\frac{N-1}N}.
\end{equation*}
Hence \eqref{eqn:sharp:lower:estimate:norm:word} holds.
Finally, it is clear that \eqref{eqn:sharp:upper:estimate:norm:word} and \eqref{eqn:sharp:lower:estimate:norm:word} give Theorem~{\ref{thm:sharp:estimate:norm:word}~\ref{thm:sharp:estimate:norm:word:b}}.
\vspace*{6pt}

\section{Proof of Theorem \ref{prop:sharp:upper:estimate:norm:word:ST}}
\label{section:proof:of:proposition4.1}

From now on the Littlewood-Paley formula \eqref{eq:LP2} we will be repeatedly used without metioning it explicitly.
Estimate \eqref{eqn:sharp:upper:estimate:norm:word:ST} is a consequence of the following fundamental result.

\begin{proposition}\label{prop:upper:estimate:norm:Q}
Let $\om\in \SW$.
 For $g\in\H(\overline{\D})$, $\sigma\in\Q$, $\sigma>0$, and $\ell\in\N$, we define
\begin{equation}\label{eqn:definition:Q}
Q_g^{\sigma,\ell}f:=|g|^{\sigma\ell}\,T^{\ell}_gf
\qquad(f\in\H(\D)).
\end{equation}
If  $0<p<\infty$, then $Q_{g}^{\sigma,\ell}$ is a bounded operator from 
$\Ap2$ to $\Lp2$ 
and its norm $\|Q_{g}^{\sigma,\ell}\|_{\Lp2}$ satisfies the estimate
\begin{equation}\label{eqn:upper:estimate:norm:Q}
\|Q_{g}^{\sigma,\ell}\|_{\Lp2}\le C\,\|g\|^{s\ell}_{\B^s_{\varphi}}, 
\end{equation}
where $s=\sigma+1$ and $C>0$ is a constant \textup{(}only depending on $\omega,p,\sigma$, and $\ell$\textup{)}.
\end{proposition}

\begin{proof}
Let us consider the irreducible fraction expression of $\sigma$, {\em i.e.} $\sigma=\frac{m}n$, where $m,n\in\N$ are coprime. 
Let $f\in \Ap2$ so that $\|f\|_{\Ap2}=1$.
Since 
$\|Q^{\sigma,\ell}_gf\|_{\Lp2}=\|g^{3m\ell}(T^{\ell}_g f)^{3n}\|^ {\frac1{3n}}_{A^{\frac{p}{3n}}_{\om_p}}$ and
\[
\bigl(g^{3m\ell}(T^{\ell}_g f)^{3n}\bigr)'
=3m\ell g^{3m\ell-1}g'(T_g^{\ell}f)^{3n}+3ng^{3m\ell}g'(T^{\ell-1}_gf)(T^{\ell}_gf)^{3n-1},
\] 
there is a constant $C_1>0$, which only depends on $\alpha,p,m$, and $n$, such that
\begin{equation}\label{eqn:estimate:norm:Qf}
\|Q^{\sigma,\ell}_gf\|_{\Lp2}\le C_1\,(A_f+B_f),
\end{equation}
where
\begin{align*} A_f&=\|g^{3m\ell-1}g'(T_g^{\ell}f)^{3n}\|^{\frac1{3n}}_{A^{\frac{p}{3n}}_{\om^{p/2}\left(1+\varphi'\right)^{-\frac{p}{3n}}}}
\qquad\mbox{and}
\\
B_f&=\|g^{3m\ell}g'(T^{\ell-1}_gf)(T^{\ell}_gf)^{3n-1}\|^{\frac1{3n}}_{A^{\frac{p}{3n}}_{\om^{p/2}\left(1+\varphi'\right)^{-\frac{p}{3n}}}}.
\end{align*}
Note that $\sigma+1=\frac{3m}{3n}+1<3m\le 3m\ell$, and so
\begin{align*}
A_f&=\||g|^{\sigma}g'|g|^{3m\ell-(\sigma+1)}(T_g^{\ell}f)^{3n}\|^{\frac1{3n}}_{L^{\frac{p}{3n}}_{\om^{p/2}\left(1+\varphi'\right)^{-\frac{p}{3n}}}}
\\
&\lesssim \|g\|^{\frac{\sigma+1}{3n}}_{\B^s_{\varphi}}\,\||g|^{3m\ell-(\sigma+1)}(T_g^{\ell}f)^{3n}\|^{\frac1{3n}}_{L^{\frac{p}{3n}}_{\om_p}}.
\end{align*}
Then we apply H\"{o}lder's inequality with exponents $\frac{3m\ell}{3m\ell-(\tau+1)}$ and $\frac{3m\ell}{\tau+1}$ to get that
\begin{align*}
A_f &\le\|g\|^{\frac{\sigma+1}{3n}}_{\B^s_{\varphi}}
    \||g|^{3m\ell}(T^{\ell}_gf)^{3n}\|^{\frac{3m\ell-(\sigma+1)}{9mn\ell}}_{L^{\frac{p}{3n}}_{\om_p}}
    \|(T^{\ell}_gf)^{3n}\|^{\frac{\sigma+1}{9mn\ell}}_{L^{\frac{p}{3n}}_{\om_p}}
\\ 
&=\|g\|^{\frac{\sigma+1}{3n}}_{\B^s_{\varphi}}
 \|Q^{\sigma,\ell}_gf\|^{\frac{3m\ell-(\sigma+1)}{3m\ell}}_{\Lp2}
 \|T^{\ell}_gf\|^{\frac{\sigma+1}{3m\ell}}_{\Ap2}.
\end{align*}
Then, 
taking into account that $\|T^{\ell}_g\|_{\Ap2}\lesssim \|g\|^{\ell}_{\B_\varphi}=\|g\|^{\ell}_{\B^1_{\varphi}}$, and Theorem~{\ref{prop:radicality:estimateintro},} we obtain that there is a constant $C_2>0$, only depending on $\alpha,p,m,n$, and $\ell$, such that 
\begin{equation}\label{eqn:estimate:A}
A_f\lesssim 
\|g\|^{\frac{\sigma+1}{3n}}_{\B^s_{\varphi}}
\|g\|^{\frac{\sigma+1}{3m}}_{\B^1_{\varphi}}
\,\|Q^{\sigma,\ell}_g\|^{1-\frac{\sigma+1}{3m\ell}}_{\Lp2}
\le C_2\,\|g\|^{\frac{(\sigma+1)^2}{3m}}_{\B^s_{\varphi}}\,\|Q^{\sigma,\ell}_g\|^{1-\frac{\sigma+1}{3m\ell}}_{\Lp2},
\end{equation}
since $\frac{\sigma+1}{3n}+\frac{\sigma+1}{3m}=\frac{(\sigma+1)^2}{3m}$.
\vspace*{6pt}

 Now let us estimate $B_f$. Since $3m\ell-\sigma>0$, we have that 
\begin{align*}
B_f
&=\||g|^{\sigma}g'|g|^{3m\ell-\sigma}(T^{\ell-1}_gf)(T^{\ell}_gf)^{3n-1}\|^{\frac1{3n}}_{L^{\frac{p}{3n}}_{\om^{p/2}\left(1+\varphi'\right)^{-\frac{p}{3n}}}}\\
&\le\|g\|^{\frac{\sigma+1}{3n}}_{\B^s_{\varphi}}\,
    \||g|^{3m\ell-\sigma}(T^{\ell-1}_gf)(T^{\ell}_gf)^{3n-1}\|^{\frac1{3n}}_{L^{\frac{p}{3n}}_{\om_p}}.
\end{align*}
But 
\[
3m\ell-\sigma=3m\ell-\sigma\ell+\sigma\ell-\sigma=\sigma\ell(3n-1)+\sigma(\ell-1)
\]
and so we may apply H\"{o}lder's inequality with exponents $3n$ and $\frac{3n}{3n-1}$ to obtain
\begin{align*}
B_f
&\le\|g\|^{\frac{\sigma+1}{3n}}_{\B^s_{\varphi}}\,
    \||g|^{\sigma(\ell-1)}(T^{\ell-1}_gf)\,
 |g|^{\sigma\ell(3n-1)}(T^{\ell}_gf)^{3n-1}\|^{\frac1{3n}}_{L^{\frac{p}{3n}}_{\om_p}}\\
&\le\|g\|^{\frac{\sigma+1}{3n}}_{\B^s_{\varphi}}\,
    \||g|^{3n\sigma(\ell-1)}(T^{\ell-1}_gf)^{3n}\|^{\frac{1}{9n^2}}_{L^{\frac{p}{3n}}_{\om-p}}\, 
    \||g|^{3n\sigma\ell}(T^{\ell}_gf)^{3n}\|^{\frac{3n-1}{9n^2}}_{L^{\frac{p}{3n}}_{\om_p}} \\
&=\|g\|^{\frac{\sigma+1}{3n}}_{\B^s_{\varphi}}\,
    \|Q_g^{\sigma,\ell-1}f\|^{\frac1{3n}}_{\Lp2}\,
 \|Q_g^{\sigma,\ell}f\|^{1-\frac1{3n}}_{\Lp2}.
\end{align*}
It follows that
\begin{equation}\label{eqn:estimate:B}
B_f\le\left\{
\begin{array}{ll}
\|g\|^{\frac{\sigma+1}{3n}}_{\B^s_{\varphi}}\,\,
 \|Q^{\sigma,1}_g\|^{1-\frac1{3n}}_{\Lp2},
 &\mbox{if $\ell=1$,}\vspace*{4pt}
\\
\|g\|^{\frac{\sigma+1}{3n}}_{\B^s_{\varphi}}\,\,
    \|Q_g^{\sigma,\ell-1}\|^{\frac1{3n}}_{\Lp2}\,\,
 \|Q^{\sigma,\ell}_g\|^{1-\frac1{3n}}_{\Lp2}, &\mbox{if $\ell>1$.}
\end{array}
\right.
\end{equation}
Therefore~{\eqref{eqn:estimate:norm:Qf}}, \eqref{eqn:estimate:A}, and~{\eqref{eqn:estimate:B}}  imply that there is a constant $C_{3,\ell}>0$, which only depends on $\alpha,p,q$ and $\ell$, such that
\begin{equation*}\begin{split}
&\|Q^{\sigma,\ell}_g\|_{\Lp2}
\\ & \le
 C_{3,\ell}\,\Bigl(\|g\|^{\frac{(\sigma+1)^2}{3m}}_{\B^s_{\varphi}}\,\|Q^{\sigma,\ell}_g\|^{1-\frac{\sigma+1}{3m\ell}}_{\Lp2}+
 \|g\|^{\frac{\sigma+1}{3n}}_{\B^s_{\varphi}}\,\,
 \|Q^{\sigma,\ell}_g\|^{1-\frac1{3n}}_{\Lp2}
\Bigr),\quad\mbox{if $\ell=1$,}
\end{split}\end{equation*}
and
\begin{equation*}\begin{split}
& \|Q^{\sigma,\ell}_g\|_{\Lp2}
\\ & \le C_{3,\ell}\,\Bigl(\|g\|^{\frac{(\sigma+1)^2}{3m}}_{\B^s_{\varphi}}\,\|Q^{\sigma,\ell}_g\|^{1-\frac{\sigma+1}{3m\ell}}_{\Lp2}+\|g\|^{\frac{\sigma+1}{3n}}_{\B^s_{\varphi}}\,\,
    \|Q_g^{\sigma,\ell-1}\|^{\frac1{3n}}_{\Lp2}\,\,
 \|Q^{\sigma,\ell}_g\|^{1-\frac1{3n}}_{\Lp2}\Bigr),
\end{split}\end{equation*}
if $\ell>1$. 

Recall that $0<\|Q^{\sigma,\ell}_g\|_{\Lp2}<\infty$, for any $\sigma$ and $\ell$, if $g$ is not constant, while $\|Q^{\sigma,\ell}_g\|_{\Lp2}=0$, for all $\sigma$ and $\ell$, otherwise. In particular, if $g$ is constant then~{\eqref{eqn:upper:estimate:norm:Q}} holds. On the other hand, when $g$ is not constant, we may divide by $C_{3,\ell}\|Q^{\sigma,\ell}_g\|_{\Lp2}$ in the preceding inequalities to get that
\begin{equation*}
C_{3,1}^{-1}\le
 \biggl(\frac{\|g\|^{\sigma+1}_{\B^s_{\varphi}}}{\|Q^{\sigma,\ell}_g\|_{\Lp2}}\biggr)^{\frac{\sigma+1}{3m}}
+\biggl(\frac{\|g\|^{\sigma+1}_{\B^s_{\varphi}}}{\|Q^{\sigma,\ell}_g\|_{\Lp2}}\biggr)^{\frac1{3n}}
\end{equation*}
and
\begin{equation*}
C_{3,\ell}^{-1}\le \biggl(\frac{\|g\|^{\sigma+1}_{\B^s_{\varphi}}}{\|Q^{\sigma,\ell}_g\|^{\frac{1}{\ell}}_{\Lp2}}\biggr)^{\frac{\sigma+1}{3m}}
+\biggl(\|g\|^{\sigma+1}_{\B^s_{\varphi}}
\frac{\|Q^{\sigma,\ell-1}_g\|_{\Lp2}}
 {\|Q^{\sigma,\ell}_g\|_{\Lp2}}\biggr)^{\frac1{3n}}
\quad(\ell>1).
\end{equation*}
 Since $\frac{\sigma+1}{3m}=\frac{\sigma+1}\sigma\frac1{3n}$, we may apply the convexity inequality 
\[
(x+y)^{\kappa}\le 2^{\kappa-1}(x^{\kappa}+y^{\kappa})
\qquad(x,y>0,\,\kappa\ge1),
\] 
to deduce that
\begin{equation}\label{eqn:inequality:ell=1}
2^{1-3n}C_{3,1}^{-3n}\le
 \biggl(\frac{\|g\|^{\sigma+1}_{\B^s_{\varphi}}}{\|Q^{\sigma,\ell}_g\|_{\Lp2}}\biggr)^{\frac{\sigma+1}{\sigma}}
+\frac{\|g\|^{\sigma+1}_{\B^s_{\varphi}}}{\|Q^{\sigma,\ell}_g\|_{\Lp2}}
\end{equation}
and
\begin{equation}\label{eqn:induction:inequality}
2^{1-3n}C_{3,\ell}^{-3n}\le \biggl(\frac{\|g\|^{\sigma+1}_{\B^s_{\varphi}}}{\|Q^{\sigma,\ell}_g\|^{\frac{1}{\ell}}_{\Lp2}}\biggr)^{\frac{\sigma+1}\sigma}
+\|g\|^{\sigma+1}_{\B^s_{\varphi}}
\frac{\|Q^{\sigma,\ell-1}_g\|_{\Lp2}}
 {\|Q^{\sigma,\ell}_g\|_{\Lp2}}
\quad(\ell>1).
\end{equation}
Now we can prove~{\eqref{eqn:upper:estimate:norm:Q}} by induction on $\ell$. First note that the case $\ell=1$ follows from~{\eqref{eqn:inequality:ell=1}}. 
Let $\ell>1$.
By the induction hypothesis, there is a constant $M>0$, which only depends on $\alpha,p,\sigma$, and $\ell-1$  such that 
\begin{equation*}
\|Q_g^{\sigma,\ell-1}\|_{\Lp2}
\le M\,\|g\|^{(\sigma+1)(\ell-1)}_{\B^s_{\varphi}}, 
\end{equation*}
for any $\ell\in\N$, $\ell\ge 2$. 
Then, by~{\eqref{eqn:induction:inequality}},
 we have that
\begin{equation*}
2^{1-3n}C_{3,\ell}^{-3n}\le \biggl(\frac{\|g\|^{(\sigma+1)\ell}_{\B^s_{\varphi}}}{\|Q^{\sigma,\ell}_g\|_{\Lp2}}\biggr)^{\frac{(\sigma+1)}{\sigma\ell}}
+M
\frac{\|g\|^{(\sigma+1)\ell}_{\B^s_{\varphi}}}
 {\|Q^{\sigma,\ell}_g\|_{\Lp2}},
\end{equation*}
Then, it follows that there exists a constant $C>0$, only depending on $\alpha,p,\sigma$, and $\ell$, such that
 \begin{equation*}
\|Q_g^{\sigma,\ell}\|_{\Lp2}
\le C\,\|g\|^{(\sigma+1)\ell}_{\B^s_{\varphi}}. 
\end{equation*}
Hence the proof is complete.
\end{proof}

\begin{proof}[{\bf Last step in the proof of Theorem~{\ref{prop:sharp:upper:estimate:norm:word:ST}}}]
Since $S^m_g T_g=\frac1{m+1}\,T_{g^{m+1}}$, 
 \eqref{eqn:sharp:upper:estimate:norm:word:ST} clearly holds for $n=1$. So we assume that $n>1$. 
By Proposition~{\ref{prop:norm:dilations}}  and \eqref{eqn:radialized:symbol:Bloch}, we may also assume that $g\in\H(\overline{\D})$. Then there exists a constant $C_1>0$, only depending on $\alpha,p,m$ and $n$, which, for any $f\in \Ap2$, satisfies that
\begin{align*}
\|S^m_gT^n_gf\|_{\Ap2}
&\le C_1\, \|g^mg'T^{n-1}_gf\|_{A^p_{\om^{p/2}\left(1+\varphi'\right)^{-p}}}
\\
&=C_1\,\||g|^{\frac{m}n}g'\,|g|^{m-\frac{m}n}T^{n-1}_gf\|_{A^p_{\om^{p/2}\left(1+\varphi'\right)^{-p}}}
\\
&\le C_1\,\|g\|^s_{\B^s_{\varphi}} 
   \,\||g|^{\frac{m}n(n-1)}\,T^{n-1}_gf\|_{\Lp2}
\\
&=C_1\,\|g\|^s_{\B^s_{\varphi}}\,\|Q^{\sigma,\ell}_gf\|_{\Lp2},
\end{align*}
where $\sigma=\frac{m}n$, $\ell=n-1$, and $Q^{\sigma,\ell}_gf$ is defined by~{\eqref{eqn:definition:Q}}. It follows that
\begin{equation*}
\label{eqn:sharp:upper:estimate:norm:SmTn:1}
\|S^m_gT^n_g\|_{\Ap2}\le C_1\,\|g\|^s_{\B^s_{\varphi}} 
   \,\|Q^{\sigma,\ell}_g\|_{\Lp2}
\qquad(g\in\H(\overline{\D})).
\end{equation*}
Now Proposition~{\ref{prop:upper:estimate:norm:Q}} shows that there is a constant $C_2>0$, which only depends on $\alpha$, $p$, $\sigma$ and $n$,  so that
\begin{equation*}
\label{eqn:sharp:upper:estimate:norm:SmTn:2}
\|Q^{\sigma,n-1}_g\|_{\Lp2}\le 
C_2\,\|g\|^{s(n-1)}_{\B^s_{\varphi}}
\qquad(g\in\H(\overline{\D})).
\end{equation*}
Therefore 
\[
\|S^m_gT^n_g\|_{\Ap2}
\le C_1C_2\,\|g\|^{sn}_{\B^s_{\varphi}}
\qquad(g\in\H(\overline{\D})),
\]
and hence Theorem~{\ref{prop:sharp:upper:estimate:norm:word:ST}} is proved. 
\end{proof}

\section{Proof of Theorem~{\ref{prop:sharp:lower:estimate:norm:word:ST}}
and Corollary~{\ref{cor:main:thm}}}
\label{section:proof:of:proposition4.2}

We prove estimate~{\eqref{eqn:sharp:lower:estimate:norm:word:ST}} in two steps as follows:

\begin{proposition}\label{prop:sharp:lower:estimate:norm:word:1}
Let $\ell\in\N$ and $\sigma\in\Q$, $\sigma>0$. If $\om\in\SW$, then

\begin{equation}
\label{eqn:lem1:sharp:lower:estimate:norm:word}
\|g\|^{s\ell}_{\B^s_{\varphi}}\lesssim
\opnorm{Q_g^{\sigma,\ell}}_{\Lp2} 
\qquad(g\in\H(\overline{\D})),
\end{equation}
where $s=\sigma+1$.
\end{proposition}

\begin{proposition}\label{lem:sharp:lower:estimate:norm:word:2}
Let  $\sigma=\frac{m}n$, $m,n\in\N$. If $\om\in\SW$, then
\begin{equation}
\label{eqn:lem2:sharp:lower:estimate:norm:word1}
\opnorm{Q_g^{\sigma,n}}_{\Lp2} 
\lesssim\opnorm{S^m_gT^n_g}_{\Ap2}
\qquad(g\in\H(\overline{\D})).
\end{equation}
\end{proposition}
It is clear that combining Propositions \ref{lem:sharp:lower:estimate:norm:word:2} (case $\ell=n$) and \ref{prop:sharp:lower:estimate:norm:word:1} we get~{\eqref{eqn:sharp:lower:estimate:norm:word:ST}}. 
We begin with a proof of Proposition \ref{lem:sharp:lower:estimate:norm:word:2}.
\begin{proof}[{\bf Proof of Proposition \ref{lem:sharp:lower:estimate:norm:word:2}}]
Let $f\in \Ap2(0)$ such that $\|f\|_{\Ap2}=1$.
Since $|Q_g^{\sigma,n}f|=|g^m\,T_g^nf|$ and  
\[
(g^m\,T_g^nf)'
=m\,g^{m-1}g'\,T_g^nf+g^m(T_g^nf)'
=m\,g^{m-1}g'\,T_g^nf+(S_g^mT_g^nf)',
\]
we have that
\begin{equation}
\label{eqn:lem2:sharp:lower:estimate:norm:word2}
\|Q_g^{\sigma,n}f\|_{\Lp2}
\lesssim\|g^{m-1}g'\,T_g^nf\|_{A^p_{\om^{p/2}(1+\varphi')^{-p}}}+\|S_g^mT_g^nf\|_{\Ap2}.
\end{equation}
If $m=1$ then Theorem~\ref{thm:compo} implies that
\begin{align*}
\|g^{m-1}g'\,T_g^nf\|_{A^p_{\om^{p/2}(1+\varphi')^{-p}}}
&\le\|g\|_{\B_{\varphi}}\,\|T_g^nf\|_{\Ap2}\\
&\lesssim \|T_g\|^{1+n}_{\Ap2}
\lesssim\opnorm{S^m_gT^n_g}_{\Ap2},
\end{align*}
so  \eqref{eqn:lem2:sharp:lower:estimate:norm:word2} gives 
\eqref{eqn:lem2:sharp:lower:estimate:norm:word1} for $m=1$.
 Thus from now on assume that $m>1$. 
Since
\begin{align*}
|g^{m-1}g'\,T_g^nf|\,(1+\varphi')^{-1}
&=(|g'| (1+\varphi')^{-1})\,|T_g^nf|^{\frac1m}\,
\Bigl(|g|^m|T_g^nf|\Bigr)^{\frac{m-1}m}\\
\nonumber
&=(|g'|(1+\varphi')^{-1})\,|T_g^nf|^{\frac1m}\,
|Q^{\sigma,n}f|^{\frac{m-1}m},
\end{align*}
 we may apply H\"older's inequality with exponents $m$ and $\frac{m}{m-1}$ and Theorem~\ref{thm:compo} to get 
\begin{align}
\label{eqn:lem2:sharp:lower:estimate:norm:word3}
\|g^{m-1}g'\,T_g^nf\|_{A^p_{\om^{p/2}(1+\varphi')^{-p}}}
&\le\|g\|_{\B_{\varphi}}\,\|T_g^nf\|_{\Ap2}^{\frac1m} \,\|Q^{\sigma,n}f\|_{\Lp2}^{\frac{m-1}m}\\
\nonumber
&\lesssim\|T_g\|_{\Ap2}^{\frac{m+n}m} \,\opnorm{Q^{\sigma,n}}_{\Lp2}^{\frac{m-1}m}\\
\nonumber
&\lesssim\opnorm{S^m_gT^n_g}_{\Ap2}^{\frac1m} \,\opnorm{Q^{\sigma,n}}_{\Lp2}^{\frac{m-1}m}.
\end{align} 
By \eqref{eqn:lem2:sharp:lower:estimate:norm:word2} and \eqref{eqn:lem2:sharp:lower:estimate:norm:word3} we obtain that
\begin{equation*}
\|Q_g^{\sigma,n}f\|_{\Lp2}
\lesssim\opnorm{S^m_gT^n_g}_{\Ap2}^{\frac1m} \opnorm{Q_g^{\sigma,n}}_{\Lp2} ^{\frac{m-1}m}
+\opnorm{S^m_gT^n_g}_{\Ap2}.
\end{equation*}
Thus, either $\opnorm{S^m_gT^n_g}_{\Ap2}=\opnorm{Q_g^{\sigma,n}}_{\Lp2}=0$
or $0<\opnorm{S^m_gT^n_g}_{\Ap2}<\infty$ and 
\begin{equation*}
 \frac{\opnorm{Q_g^{\sigma,n}}_{\Lp2} }{\opnorm{S^m_gT^n_g}_{\Ap2}} \lesssim 
 1+\left( \frac{\opnorm{Q_g^{\sigma,n}}_{\Lp2} }{\opnorm{S^m_gT^n_g}_{\Ap2}}\right)^{\frac{m-1}m}. 
\end{equation*}
Therefore \eqref{eqn:lem2:sharp:lower:estimate:norm:word1} holds and that ends the proof of the proposition.
\end{proof}
\subsection{Proof of Proposition \ref{prop:sharp:lower:estimate:norm:word:1}}
The proof is by induction on $\ell$. Since the proof is lengthy, we split it into two propositions. The case $\ell=1$ is just the following proposition, and the induction step will be done in 
Proposition~{\ref{lem:sharp:lower:estimate:norm:word:4}.}

\begin{proposition}\label{lem:sharp:lower:estimate:norm:word:3}
Let $\sigma\in\Q$, $\sigma>0$, and let $\om=e^{-2\varphi}\in \SW$. Then
\begin{equation}
\label{eqn:lem3:sharp:lower:estimate:norm:word}
\|g\|^s_{\B^s_{\varphi}}\lesssim
\opnorm{Q_g^{\sigma,1}}_{\Lp2} 
\qquad(g\in\H(\overline{\D})),
\end{equation}
where $s=\sigma+1$.
\end{proposition}

The proof of Proposition \ref{lem:sharp:lower:estimate:norm:word:3} requires the following technical lemmas.
Recall that $K^\om_a$ denotes the Bergman reproducing kernel of $A^2_\om$ at  the point $a$ (see \ref{subsec:Bergman:kernels}). 

\begin{lemma}
Let $\om=e^{-2\varphi}\in \SW$, $0<p<\infty$, and $\b,N\ge 0$. Then:
\begin{gather}
\label{eqn1:le:estimates}
\sup_{z\in\D} |K^\om_a(z)|\, \frac{\om(z)^{1/2}\,\left(1+\varphi'(z)\right)^{N}}{\tau(z)^{\b}}
\simeq  \frac{ \left(1+\varphi'(a\right)^{N}}{ \om(a)^{1/2}\tau(a)^{\b+2}} \qquad(a\in\D).
\\
\label{eqn2:le:estimates}
\|K^\om_a\|^p_{A^p_{\om^{p/2}\left(1+\varphi'\right)^\beta}}
\simeq 
\frac{\left(1+\varphi'(a)\right)^\beta} {\om(a)^{p/2}\tau(a)^{2p-2}}\qquad(a\in\D). 
\end{gather}
\end{lemma}
\begin{proof}
Take $\delta>0$ small enough such that \eqref{eqn:tau:1+varphiprime:on:Ddeltaa} and \eqref{5.3bis} hold.
Then \eqref{5.3tris} directly gives the estimates
\begin{gather*}
\sup_{z\in D_{\delta}(a)} |K^\om_a(z)|\,
\frac{\om(z)^{1/2}\,\left(1+\varphi'(z)\right)^{N}} {\tau(z)^{\b}} \simeq 
 \frac{\left(1+\varphi'(a)\right)^{N}}
{\om(a)^{1/2}\,\tau(a)^{\b+2}} \quad\mbox{and}
\\
\int_{D_{\delta}(a)}|K^\om_a|^p\om^{p/2}\left(1+\varphi'\right)^\beta\,dA
\simeq
\frac{\left(1+\varphi'(a)\right)^\beta}{\om(a)^{p/2}\tau(a)^{2p-2}},
\qquad\mbox{for $a\in\D$,}
\end{gather*}
so \eqref{eqn1:le:estimates} and \eqref{eqn2:le:estimates} will be proved once we have shown the upper  estimates
\begin{align}
\label{eqn3:le:estimates}
\sup_{z\in D_{\delta}(a)^c} \Phi(z)
&\lesssim
 \frac{\left(1+\varphi'(a)\right)^{N}}
{\om(a)^{1/2}\,\tau(a)^{\b+2}}
\qquad(a\in\D)
\\
\label{eqn4:le:estimates}
\int_{D_{\delta}(a)^c}\Psi\,dA
&\lesssim
\frac{\left(1+\varphi'(a)\right)^\beta}{\om(a)^{p/2}\tau(a)^{2p-2}}
\qquad(a\in\D),
\end{align}
where $\Phi:= |K^\om_a|\,
\frac{\om^{1/2}\,\left(1+\varphi'\right)^{N}} {\tau^{\b}}$, $\Psi:=|K^\om_a|^p\om^{p/2}\left(1+\varphi'\right)^\beta$, and $D_{\delta}(a)^c:=\D\setminus D_{\delta}(a)$.
In order to prove these estimates,  for every $a\in\D$, we consider the partition of $D_{\delta}(a)^c$ by the regions $R(a):=\{z\in D_{\delta}(a)^c:|z|\le|a|\}$ and $R'(a):=D_{\delta}(a)^c\setminus R(a)$.
First, note that, since $\varphi'$ is non-decreasing and $\tau$ is decreasing on $[0,1)$ \eqref{eqn:estimate:weighted:Linfty:norm:Bergman:kernel} and \eqref{eqn:estimate:Apomegap:norm:Bergman:kernel} give:
\begin{align}
\label{eqn5:le:estimates}
\sup_{z\in R(a)}\Phi(z)
&\lesssim \frac{\left(1+\varphi'(a)\right)^{N}}
                     {\om(a)^{1/2}\,\tau(a)^{\b+2}}
\qquad(a\in\D).
\\
\label{eqn6:le:estimates}
\int_{R(a)}\Psi\,dA
&\lesssim\frac{\left(1+\varphi'(a)\right)^\beta}
{\om(a)^{p/2}\tau(a)^{2p-2}}
\quad(a\in\D).
\end{align}
Let $\eta>0$ so that \ref{rapidly:decreasing:condition:e} and \eqref{5.3} hold and  
$M>\max\{\beta+N\eta+1,1+\tfrac{\eta\beta}{p}, \tfrac{2}{p}\}$. 
Then, taking also into account \eqref{5.4}, we have that
\begin{align*}
\Phi(z)
& \lesssim \frac{\left(1+\varphi'(z)\right)^{N}}{\om(a)^{1/2}\,\tau(a)\tau(z)^{\b+1} }   \Big( \frac{\min( \tau(a), \tau(z) }{|a-z|} \Big)^{M}
\\ 
& \lesssim   
\frac{\left(1+\varphi'(z)\right)^{N}\bigl(\min(\tau(a),\tau(z)\bigr)^{M}}
      {\om(a)^{1/2}\,\tau(a)^{1+M}\,\tau(z)^{\b+1}}
\le 
\frac{\left(1+\varphi'(z)\right)^{N}\tau(z)^{M-\b-1}}
      {\om(a)^{1/2}\,\tau(a)^{1+M}} 
\\ 
&= 
\frac{\left((1+\varphi'(z))\tau^{\eta}(z)\right)^{N}\tau(z)^{M-\b-N\eta-1}}
      {\om(a)^{1/2}\,\tau(a)^{1+M}} 
\le
\frac{\left(1+\varphi'(a)\right)^{N}}
      {\om(a)^{1/2}\,\tau(a)^{2+\b}},\mbox{ for $z\in R'(a)$}. 
\end{align*}
This estimate together with \eqref{eqn5:le:estimates} proves \eqref{eqn3:le:estimates}. In order to give an upper bound of $\int_{R'(a)}\Psi\,dA$ we partition $R'(a)$ by the regions
\begin{equation*}
R'_k(a):=\{z\in\D: |a|<|z|,\,2^{k-1}\delta\tau(a)\le |z-a|< 2^{k}\delta\tau(a)\}
\qquad(k\in\N).
\end{equation*} 
Then, by \eqref{5.3}, \eqref{5.4} and condition \ref{rapidly:decreasing:condition:e}, we have that
\begin{align*}
\Psi(z)
&\lesssim   
\frac{\bigl(\min(\tau(a),\tau(z))\bigr)^{Mp}}{\om(a)^{p/2}\,\tau(a)^{p}\,\tau(z)^{p}\,|a-z|^{Mp}}
\left(1+\varphi'(z)\right)^\beta
\\ 
&\simeq 
\frac1{2^{kMp}}\,
\frac{\bigl(\min(\tau(a),\tau(z))\bigr)^{Mp}}
{\om(a)^{p/2}\,\tau(a)^{(M+1)p}\,\tau(z)^{p}}
\left(1+\varphi'(z)\right)^\beta
\\
&\le 
\frac1{2^{kMp}}\,
\frac{\bigl(\min(\tau(a),\tau(z))\bigr)^{(M-1)p}}
{\om(a)^{p/2}\,\tau(a)^{(M+1)p}}
\left(1+\varphi'(z)\right)^\beta
\\
&\le 
\frac1{2^{kMp}}\,
\frac{\bigl(\min(\tau(a),\tau(z))\bigr)^{(M-1)p-\eta\beta}}
{\om(a)^{p/2}\,\tau(a)^{(M+1)p}}
\left(\bigl(1+\varphi'(z)\bigr)\tau^{\eta}(z)\right)^\beta
\\
&\lesssim 
\frac1{2^{kMp}}\,
\frac{\bigl(\min(\tau(a),\tau(z))\bigr)^{(M-1)p-\eta\beta}}
{\om(a)^{p/2}\,\tau(a)^{(M+1)p-\eta\beta}}
\bigl(1+\varphi'(a)\bigr)^\beta
\\
&\le 
\frac1{2^{kMp}}\,
\frac{\bigl(1+\varphi'(a)\bigr)^\beta}
{\om(a)^{p/2}\,\tau(a)^{2p}},\quad\mbox{for $a\in\D$, $z\in R'_k(a)$ and $k\in\N$,} 
\end{align*}
and so
\begin{equation*}
\int_{R'_k(a)}\Psi\,dA
\lesssim  \frac1{2^{k(Mp-2)}}\,
\frac{\bigl(1+\varphi'(a)\bigr)^\beta}
{\om(a)^{p/2}\,\tau(a)^{2p-2}}
\qquad(a\in\D,k\in\N).
\end{equation*}
Therefore  
\begin{equation}\label{eqn7:le:estimates}
\int_{R'(a)}\Psi\,dA=\sum_{k=1}^{\infty}\int_{R'_k(a)}\Psi\,dA
\lesssim \frac{\bigl(1+\varphi'(a)\bigr)^\beta}
{\om(a)^{p/2}\,\tau(a)^{2p-2}}\qquad(a\in\D).
\end{equation}
Hence \eqref{eqn4:le:estimates} directly follows from \eqref{eqn6:le:estimates} and  \eqref{eqn7:le:estimates}.
\end{proof}

\begin{lemma}\label{le:crecimientog-gprima}
Let $\alpha>0$ and let $\psi$ be a function satisfying \eqref{eq:compatibility} and \eqref{eq:cond-crecimiento}. Then
\begin{equation}\label{eqn:derivative:lesssim:function}
\sup_{0\le r<1} \frac{M_\infty(r,g')}{\psi'(r)^{\alpha+1}}
\lesssim
\sup_{0\le r<1}\frac{M_\infty(r,g)}{\psi'(r)^{\alpha}}
\qquad(g\in\H(\D)).
\end{equation}
\end{lemma}
\begin{proof}
By \eqref{eq:Cauchy} with $\rho=\rho(r)=r+\frac{\delta}{\psi'(r)}$ and \eqref{eq:cond-crecimiento}, we have the estimate
\begin{equation*}
 \frac{M_\infty(r,g')}{\psi'(r)^{\alpha+1}} 
\le \frac{C}{\delta} \frac{M_\infty\left( r+\frac{\delta}{\psi'(r)},g \right)}{\psi'(r)^{\alpha}}
\lesssim \frac{M_\infty\left( r+\frac{\delta}{\psi'(r)},g \right)}{\psi'\left( r+\frac{\delta}{\psi'(r)}\right)^{\alpha}}
 \le \sup_{0\le s<1}\frac{M_\infty(s,g)}{\psi'(s)^{\alpha}}, 
\end{equation*}
for $0\le r<1$.
This finishes the proof.
\end{proof}

\begin{lemma}
Let $\om=e^{-2\varphi}\in\SW$, $\sigma=\frac{m}{n}$, where $m,n\in\N$, and $s=\sigma+1$. Then 
\begin{equation}
\label{eqn:estimate:second:derivative}	
\sup_{z\in\D}
\frac{|g(z)|^m\,|g'(z)|^{n-1}\,|g''(z)|}
{\left(1+\varphi'(z)\right)^{n+1}}\lesssim \|g\|^{sn}_{ \B^s_{\varphi}}\qquad(g\in\H(\D)). 
\end{equation}
\end{lemma}
\begin{proof}
Note that $\|g\|^{sn}_{ \B^s_{\varphi}}\simeq\|G\,(\psi')^{-n}\|_{L^{\infty}(\D)}$, where $\psi(r)=r+\varphi(r)$ and $G=g^m(g')^n$. Then $G\in\H(\D)$ and 
$G'=mg^{m-1}(g')^{n+1}+ng^m(g')^{n-1}g''$, so 
\begin{equation}\label{eqn:estimate:second:derivative:1}
\frac{|g|^m|g'|^{n-1}|g''|}{(\psi')^{n+1}}
\le\sigma
\left(\frac{|g|^{m-1}|g'|^{n+1}}{(\psi')^{n+1}}
                                    +\frac{|G'|}{(\psi')^{n+1}}\right)
=\sigma(G_1+G_2).   
\end{equation}
Since 
$G_1=\Bigl(\frac{|g|^{\frac{m}n}|g'|}{\psi'}\Bigr)^{n-\frac{n}m}
\Bigl(\frac{|g'|}{\psi'}\Bigr)^{1+\frac{n}m}
\lesssim\bigl(\frac1{\psi'}|\nabla|g|^s|\bigr)^{n-\frac{n}m}
\Bigl(\frac{|g'|}{\psi'}\Bigr)^{1+\frac{n}m}$,
Theorem~{\ref{prop:radicality:estimateintro}} shows that
\begin{equation}\label{eqn:estimate:second:derivative:2}
\sup_{z\in\D}G_1(z)
\le\|g\|_{ \B^s_{\varphi}}^{s(n-\frac{n}m)}
\|g\|_{\B^1_{\varphi}}^{1+\frac{n}m}
\lesssim\|g\|_{\B^s_{\varphi}}^{sn}
\quad(g\in\H(\D)),
\end{equation}
because $s(n-\frac{n}m)+1+\frac{n}m=sn+1+\frac{n}m(1-s)=sn$. Moreover,
bearing in mind Propositions~\ref{prop:properties:varphi} and~\ref{prop:properties:varphi2}
and applying \eqref{eqn:derivative:lesssim:function} to $g=G$  and $\alpha=n$, we get
\begin{equation}\label{eqn:estimate:second:derivative:3}
\sup_{z\in\D}G_2(z)\lesssim \sup_{z\in\D}\frac{|G(z)|}{\psi'(z)^n}
\simeq \|g\|^{sn}_{ \B^s_{\varphi}}
\qquad(g\in\H(\D)).
\end{equation}
Therefore \eqref{eqn:estimate:second:derivative} directly follows from \eqref{eqn:estimate:second:derivative:1}, \eqref{eqn:estimate:second:derivative:2}, and 
\eqref{eqn:estimate:second:derivative:3}.
\end{proof}

\begin{lemma}
\label{lemma:estimate:gamma,p:alpha,q;p<q}
Let $\om\in\SW$ and $0<p\le q<\infty$. Then
\begin{equation}\label{eqn:estimate:gamma,1:alpha,p}
\|f\|_{A^q_{\om^{\frac{q}2}\tau^{2\left(\frac{q}{p} -1\right)}}}
\lesssim\|f\|_{\Ap2}
\qquad(f\in\H(\D)).
\end{equation}
\end{lemma}
\begin{proof}
By \eqref{eqn:growth:of:Apomegap:functions} we have that
\[
|f(z)|^q= |f(z)|^p |f(z)|^{q-p}
\lesssim 
\|f\|_{\Ap2}^{q-p}  
\tau(z)^{2\left(1-\frac{q}{p}\right)}\om(z)^{\frac{p-q}2}|f(z)|^p
\quad(z\in\D),
\]
from which \eqref{eqn:estimate:gamma,1:alpha,p} directly follows.
\end{proof}

Next lemma is an easy application of the maximum modulus principle, so we omit its proof.

\begin{lemma}
Let $\om=e^{-2\varphi}\in\SW$ and
 $N>0$. Then
\begin{equation}\label{eqn:multiply:by:z}
 \sup_{z\in\D}  
 \frac{\om(z)\tau(z)^2|f(z)|}{\left(1+\varphi'(z)\right)^{N}} 
\simeq 
\sup_{z\in\D}  \frac{\om(z)\tau(z)^2|\widetilde{f}(z)|}{\left(1+\varphi'(z)\right)^{N}}\qquad(f\in\H(\D)),
\end{equation}
where $\widetilde{f}(z)=zf(z)$.
\end{lemma}

\begin{lemma}
\label{lem:estimate:integrals}
Let $\om=e^{-2\varphi}\in\SW$,  $N\ge 0$, $p>0$ and  $\sigma\in\Q$, $\sigma>0$.
Then there is a constant $C>0$ such that, for any $g\in\H(\D)$ and $a\in\D$, we have  
\begin{equation}
\label{eqn:estimate:integrals}
\int_{\D}|Q_g^{\sigma,1}K^\om_{a,0}|\,|K^\om_a|\, \left(1+\varphi'\right)^{N} \om\,dA
\le C\, \frac{\opnorm{Q_g^{\sigma,1}}_{\Lp2}  \left(1+\varphi'(a)\right)^{N}}{\om(a)\tau(a)^2},
\end{equation}
where $K_{a,0}^\om(z):=z\,K_a^\om(z)$.
\end{lemma}

\begin{proof}
Denote by $I(a)$ the integral in \eqref{eqn:estimate:integrals} and consider two cases: 
\begin{enumerate}[label={\sf\alph*)},topsep=3pt, 
leftmargin=0pt, itemsep=4pt, wide, listparindent=0pt, itemindent=6pt] 

\item $1<p<\infty$: Let $q$ be the conjugate exponent of $p$. Then H\"{o}lder's inequality,  \eqref{eqn:estimate:Apomegap:norm:Bergman:kernel} and \eqref{eqn2:le:estimates} prove \eqref{eqn:estimate:integrals} in this case:
\begin{align*}
I(a)
&\le 
\|Q_g^{\sigma,1}K^{\om}_{a,0}\|_{\Lp2} 
\|K^\om_a\|_{A^{q}_{\om^{q/2}(1+\varphi')^{Nq}}}
\\
&\le\opnorm{Q_g^{\sigma,1}}_{\Ap2}\|K^\om_{a} \|_{\Ap2}
\|K^\om_a\|_{A^{q}_{\om^{q/2}(1+\varphi')^{Nq}}}
\\
 & \lesssim 
\frac{\opnorm{Q_g^{\sigma,1}}_{\Lp2}   \left(1+\varphi'(a)\right)^{N}}{\om(a)\tau(a)^2}.
\end{align*} 

\item  $0<p\le 1$: Then
\[
I(a)\le  M_a \int_{\D}|Q_g^{\sigma,1}K^\om_{a,0}|\, \om^{\frac12} \tau^{2\left(\frac{1}{p}-1 \right)}\,dA \qquad(a\in\D),
\]
where $M_a:=\sup_{z\in\D}\,|K^\om_a(z)|\,  \left(1+\varphi'(z)\right)^{N}\om(z)^{\frac12}\tau(z)^{-2\left(\frac{1}{p}-1 \right)}$.
Let $\frac{m}{n}$ be the irreducible fraction expression of $\sigma$. 
Then $\om^n=e^{-2n\varphi}\in\SW$ and  $(\Delta(n\varphi))^{-1/2}\simeq\tau$, so
Lemma~\ref{lemma:estimate:gamma,p:alpha,q;p<q}, \eqref{eqn2:le:estimates}, and \eqref{eqn:estimate:Apomegap:norm:Bergman:kernel}
give that
\begin{equation*}\begin{split}
\int_{\D}|Q_g^{\sigma,1}K^\om_{a,0}|\,\om^{1/2} \tau^{2\left(\frac{1}{p}-1 \right)}\,dA
&= \| g^m  (T_g K^\om_{a,0})^n \|^{1/n}_{A^{1/n}_{(\om^n)^{\frac1{2n}}\tau^{2\left(\frac{1}{p}-1\right)}}}
\\ & \lesssim   \| g^m  (T_g K^\om_{a,0})^n \|^{1/n}_{A^{p/n}_{\om^{p/2}}}=\|Q_g^{\sigma,1}K^\om_{a,0} \|_{\Lp2}
\\ & \le  \opnorm{Q_g^{\sigma,1}}_{\Lp2} \om(a)^{-\frac12}\tau^{\frac2p-2}(a),\quad\mbox{for $a\in\D$.}
\end{split}\end{equation*}
Moreover, \eqref{eqn1:le:estimates} implies that
$M_a\lesssim  \om(a)^{-\frac12}\tau^{-\frac2p}(a)  \left(1+\varphi'(a)\right)^{N}$, for $a\in\D$,
and therefore \eqref{eqn:estimate:integrals} also holds in this case.
Hence the proof is complete.\qedhere
\end{enumerate}
\end{proof}

\begin{proof}[{\bf Proof of Proposition \ref{lem:sharp:lower:estimate:norm:word:3}}]
Fix $a\in\D$. Let  $K_{a,0}^\om$ be defined as in the statement of Lemma~\ref{lem:estimate:integrals}.
Let $\sigma=\frac{m}n$ be the irreducible fraction expression of $\sigma$. Take $N=3n$ and observe that the estimate \eqref{eqn:estimate:Apomegap:norm:Bergman:kernel} implies that
\[
K^\om_a(a)=\|K_a^{\omega}\|^2_{A^2_\om}\simeq \frac1{\om(a)\tau^2(a)}
\qquad(a\in\D),
\] 
so
\begin{equation*}
\|g\|^{sN}_{ \B^s_{\varphi}}
\simeq\sup_{a\in\D}\frac{|g(a)|^{N\sigma}|g'(a)|^N} 
{(1+ \varphi'(a))^N}
\simeq\sup_{a\in\D}\frac{\om(a) \tau^{2}(a) |G_a(a)|}{(1+ \varphi'(a))^N}
\,\,(a\in\D,\,g\in\H(\overline{\D}),
\end{equation*}
where $G_a:=g^{N\sigma}(g')^NK_a^\om\in\H(\D)$.
 Therefore~{\eqref{eqn:multiply:by:z}} gives that
\begin{equation}
\label{eqn:lem3:sharp:lower:estimate:norm:word:1}
\|g\|^{sN}_{ \B^s_{\varphi}}
\simeq
\sup_{a\in\D} 
 \frac{\om(a)\tau(a)^2|\widetilde{G}_a(a)|}{(1+\varphi'(a))^N}
\qquad(a\in\D,\,g\in\H(\overline{\D})),
\end{equation}
where $\widetilde{G}_a(z):=zG_a(z)=g(z)^{N\sigma}g'(z)^NK^\om_{a,0}(z)$.
 Since $K^\om_a$ is the reproducing kernel for $A^2_\om$ at the point $a$ and $\widetilde{G}_a\in\H(\overline{\D})$, we have
\begin{equation*}
\label{eqn:lem3:sharp:lower:estimate:norm:word:2}
\widetilde{G}_a(a)
=\int_{\D}\widetilde{G}_a\,\overline{K^\om_a}\,\omega\,dA
=\lim_{r\to1^{-}}
   \int_{D_r}\widetilde{G}_a\,\overline{K^\om_a}\,\omega\,dA,
\end{equation*}
where $D_r=D(0,r)$.
Note that
$\widetilde{G}_a\,\overline{K^\om_a}\,\omega=F_1\,
\frac{\partial F_2}{\partial z}
=\tfrac{\partial\,}{\partial z}(F_1F_2)-F_2\tfrac{\partial F_1}{\partial z}$, where
$F_1:=g^{N\sigma}(g')^{N-1}\om$ and 
$F_2:=(T_g  K^\om_{a,0})\overline{K^\om_a}$ are $C^1$ functions on $\D$.
Then, since $\lim_{|z|\to1^{-}}F_1(z)F_2(z)=0$, Stokes' theorem shows that
\begin{equation*}
\widetilde{G}_a(a)
=\lim_{r\to1^{-}}
   \int_{D_r}\widetilde{G}_a\,\overline{K^\om_a}\,\omega\,dA
=-\lim_{r\to1^{-}}
\int_{D_r}F_2\tfrac{\partial F_1}{\partial z}\,dA.
\end{equation*}
Since
\begin{align*}
\biggl|\int_{D_r}F_2\tfrac{\partial F_1}{\partial z}\,dA\biggr|
&\le N\sigma\int_{\D}|g|^{N\sigma-1}|g'|^N
                               |T_gK^\om_{a,0}||K^\om_a|\,\omega\,dA\\
&\,\,\,\,\,\,\,+(N-1)\int_{\D}|g|^{N\sigma}|g'|^{N-2}|g''|
                             |T_gK^\om_{a,0}|\,|K^\om_a|\,\omega\,dA\\
&\,\,\,\,\,\,\,+\int_{\D}|g|^{N\sigma}|g'|^{N-1}
                 |T_gK^\om_{a,0}|\,|K^\om_a|\,
                 |\tfrac{\partial\omega}{\partial z}|\,dA\\
&= N\sigma\,A_1+(N-1)\,A_2+A_3,
\end{align*}
in order to complete the proof it is enough to show that
\begin{equation}
\label{eqn:lem3:sharp:lower:estimate:norm:word:4}
A_j
\lesssim \|g\|_{ \B^s_{\varphi}}^{s(N-1)}
\int_{\D}|Q_g^{\sigma,1}K^\om_{a,0}|\,|K^\om_a|\, (1+\varphi')^{N} \om\,dA
\qquad(a\in\D),
\end{equation}
for $j=1,2,3$.
Note that \eqref{eqn:lem3:sharp:lower:estimate:norm:word:1}-\eqref{eqn:lem3:sharp:lower:estimate:norm:word:4} together with \eqref{eqn:estimate:integrals} will show \eqref{eqn:lem3:sharp:lower:estimate:norm:word}, which will end the proof of the proposition.

We begin by proving \eqref{eqn:lem3:sharp:lower:estimate:norm:word:4} for $j=1$.
 Indeed, 
\begin{equation*}
A_1\le
\int_{\D}\biggl(\frac{|g|^{\sigma}|g'|}{1+\varphi'}\biggr)^{\alpha}
              \biggl(\frac{|g'|}{1+\varphi'}\biggr)^{\beta}
              |Q_g^{\sigma,1}K^\om_{a,0}|\,|K^\om_{a}|\, (1+\varphi')^{\gamma}\,\omega \,dA, 
\end{equation*}
where $\alpha,\beta,\gamma$ are the solutions of the linear system
\[
\left\{
\begin{array}{lcl}
N\sigma-1 & = & \sigma\alpha+\sigma\\
N              & = & \alpha+\beta\\
0              & = & \alpha+\beta-\gamma,
\end{array}
\right.
\]
namely, $\alpha=N-\frac1{\sigma}-1$, $\beta=\frac1{\sigma}+1$, and $\gamma=N$. Therefore, taking into account  Theorem~{\ref{prop:radicality:estimateintro}}, we have that 
 \begin{align*}
 	A_1
 	&\lesssim \|g\|_{ \B^s_{\varphi}}^{s\alpha}\,\|g\|_{ \B^1_\varphi}^{\beta}
 	\int_{\D}
 	|Q_g^{\sigma,1}K^\om_{a,0}|\,|K^\om_{a}|\, (1+\varphi')^{N} \omega \,dA\\
 	&\lesssim\|g\|_{ \B^s_{\varphi}}^{s\alpha+\beta}
 	\int_{\D}
 	|Q_g^{\sigma,1}K^\om_{a,0}|\,|K^\om_{a}|\, (1+\varphi')^{N} \omega \,dA,
 \end{align*}
which shows that \eqref{eqn:lem3:sharp:lower:estimate:norm:word:4}
holds for $j=1$, because $s\alpha+\beta=s(N-1)$.
 Since 
\begin{equation*}
 A_2\le
 \int_{\D}\biggl(\frac{|g|^m|g'|^{n-1}|g''|}{(1+\varphi')^{n+1}} \biggr)\biggl(\frac{|g|^{\sigma}|g'|}{1+\varphi'} \biggr)^{\alpha}
 |Q_g^{\sigma,1}K^\om_{a,0}|\,|K^\om_{a}|\, (1+\varphi')^{N} \omega \,dA,
 \end{equation*}
 where $\alpha=N-n-1$, we may apply \eqref{eqn:estimate:second:derivative} to get \eqref{eqn:lem3:sharp:lower:estimate:norm:word:4} for $j=2$: 
\begin{align*}
 A_2
 &\lesssim \|g\|_{ \B^s_{\varphi}}^{s(n+\alpha)}
 \int_{\D}
 |Q_g^{\sigma,1}K^\om_{a,0}|\,|K^\om_{a}|\,(1+\varphi')^{N} \omega \,dA\\
 &=\|g\|_{\B^s_{\varphi}}^{s(N-1)}
 \int_{\D}
|Q_g^{\sigma,1}K^\om_{a,0}|\,|K^\om_{a}|\,(1+\varphi')^{N} \omega \,dA.
 \end{align*}
Finally, we check \eqref{eqn:lem3:sharp:lower:estimate:norm:word:4} for $j=3$:
Just note that $\left|\frac{\partial\omega}{\partial z}\right|=2\om\varphi'$, so
\begin{align*}
A_3
&\le\int_{\D}
       |g|^{N\sigma}|g'|^{N-1}|T_gK^\om_{a,0}|\,|K^\om_a|\,
                                    \bigl|\tfrac{\partial\omega}{\partial z}\bigr|\,dA
\\ 
&=\int_{\D}
       |g|^{N\sigma}|g'|^{N-1}|T_g  K^\om_{a,0}|\,|K^\om_a|\,\om \,\varphi'dA
\\ 
&\le\int_{\D} \biggl(\frac{|g|^{\sigma}|g'|}{1+\varphi'}\biggr)^{N-1}  |Q_g^{\sigma,1}K^\om_{a,0}|\,|K^\om_{a}|\, (1+\varphi')^{N} \omega \,dA,
\\ &\le \|g\|_{ \B^s_{\varphi}}^{s(N-1)}\int_{\D}|Q_g^{\sigma,1}K^\om_{a,0}|\,|K^\om_{a}|\,(1+\varphi')^{N} \omega \,dA.\qedhere
\end{align*}
\end{proof}
The induction step for the proof of Proposition \ref{prop:sharp:lower:estimate:norm:word:1} is done in the following result.

\begin{proposition}\label{lem:sharp:lower:estimate:norm:word:4}
Let $\omega\in \SW$ and $0<p<\infty$.
Let  $\ell\in\N$,  $\sigma\in\Q$, $\sigma>0$, and $s=\sigma+1$. Assume that
\begin{equation}
\label{eqn:lem4:sharp:lower:estimate:norm:word1}
\|g\|_{\B^s_{\varphi}}^{s\ell}\lesssim
 \opnorm{Q^{\sigma,\ell}_g}_{\Lp2}  
 \qquad(g\in\H(\overline{\D})).
\end{equation}
Then
\begin{equation}
\label{eqn:lem4:sharp:lower:estimate:norm:word2}
\|g\|_{\B^s_{\varphi}}^{s(\ell+1)}\lesssim
\opnorm{Q_g^{\sigma,\ell+1}}_{\Lp2} 
\qquad(g\in\H(\overline{\D})).
\end{equation}
\end{proposition}

\begin{proof}
Let $N=kn$, with $k\in\N$ large enough. Let $f\in \Ap2(0)$ with $\|f\|_{\Ap2}=1$. Then
\begin{equation*}
\|Q_g^{\sigma,\ell}f\|_{\Lp2}^N=    
\||g|^{\sigma\ell}\,T_g^{\ell}f\|_{\Lp2}^N=
\|g^{N\sigma\ell}\,(T_g^{\ell}f)^N\|_{A^{\frac{p}N}_{\om_p}}.
\end{equation*}
Now 
\begin{align*}
\bigl(g^{N\sigma\ell}\,(T_g^{\ell}f)^N\bigr)'  
&=N\sigma\ell\, g' g^{N\sigma\ell-1}(T_g^{\ell}f)^N
+Ng^{N\sigma\ell} g' (T_g^{\ell}f)^{N-1}\, T_g^{\ell-1}f\\
&=N\sigma\ell\,(T_g^{\ell+1}f)' g^{N\sigma\ell-1}(T_g^{\ell}f)^{N-1}\\
&\,\,\,\,+N(T_g^{\ell+1}f)' g^{N\sigma\ell}  (T_g^{\ell}f)^{N-2}\, T_g^{\ell-1}f,
\end{align*}
so $\|Q_g^{\sigma,\ell}f\|_{\Lp2}^N \lesssim A+B$, where 
\begin{align*}
A&=\|(T_g^{\ell+1}f)' g^{N\sigma\ell-1}(T_g^{\ell}f)^{N-1}    \|_{A^{\frac{p}N}_{W_{p,N}}},\\
B&=\|(T_g^{\ell+1}f)' g^{N\sigma\ell}  (T_g^{\ell}f)^{N-2}\, T_g^{\ell-1}f   \|_{A^{\frac{p}N}_{W_{p,N}}},
\end{align*}
and $W_{p,N}=\om^{p/2}\left(1+\varphi'\right)^{-\frac{p}{N}}$.
We will show that 
\begin{equation}
\label{eqn:lem4:sharp:lower:estimate:norm:word3}
 A+B\lesssim \opnorm{Q_g^{\sigma,\ell+1}}_{\Lp2}\,
  \|g\|_{\B^s_{\varphi}}^{s\ell(N-\frac{\ell+1}{\ell})}. 
\end{equation}
\noindent{\sf Estimate of $A$:}
\begin{align*}
A
&\lesssim\|\bigl(T_g^{\ell+1}f g^{N\sigma\ell-1}(T_g^{\ell}f)^{N-1}    \bigr)'\|_{ A^{\frac{p}N}_{W_{p,N}} } 
  +\|T_g^{\ell+1}f \bigl(g^{N\sigma\ell-1}(T_g^{\ell}f)^{N-1}    \bigr)'\|_{A^{\frac{p}N}_{W_{p,N}} } \\
&\lesssim\|T_g^{\ell+1}f g^{N\sigma\ell-1}(T_g^{\ell}f)^{N-1}    \|_{A^{\frac{p}N}_{\om_p}}  
  +\|(T_g^{\ell+1}f) g'g^{N\sigma\ell-2}(T_g^{\ell}f)^{N-1}    \|_{A^{\frac{p}N}_{W_{p,N}}}  \\
&\,\,\,  +\|T_g^{\ell+1}f g'g^{N\sigma\ell-1}(T_g^{\ell}f)^{N-2}T_g^{\ell-1}f    \|_{A^{\frac{p}N}_{W_{p,N}}}=A_1+A_2+A_3.
\end{align*}
{\sf Estimate of $A_1$:}
\begin{align*}
|T_g^{\ell+1}f| |g|^{N\sigma\ell-1}|T_g^{\ell}f|^{N-1}
&=\bigl(|g|^{\sigma(\ell+1)}|T_g^{\ell+1}f|\bigr)\,
\bigl(|g|^{\sigma\ell}|T_g^{\ell}f|\bigr)^a\,|T_g^{\ell}f|^b
\\
&=|Q_g^{\sigma,\ell+1}f|\,|Q_g^{\sigma,\ell}f|^a|T_g^{\ell}f|^b,
\end{align*}
where $a$ and $b$ are the solutions of the system
\[
\left\{
\begin{array}{lcl}
N\sigma\ell-1 & = & \sigma(\ell+1)+a\sigma\ell\\
N-1                & = & a+b,
\end{array}
\right.
\]
namely, $a=N-1-\frac{s}{\sigma\ell}$ and $b=\frac{s}{\sigma\ell}$. Since $1+a+b=N$, we may apply H\"{o}lder's inequality with exponents $p_1=N$, $p_2=\frac{N}a$ and $p_3=\frac{N}b$  and Theorem~{\ref{prop:radicality:estimateintro}}  to get 
\begin{align*}
 A_1
 &\le \|Q_g^{\sigma,\ell+1}f\|_{\Lp2}\, 
      \|Q_g^{\sigma,\ell}f\|^a_{\Lp2}\,
      \|T_g^{\ell}f\|^b_{\Ap2} \
\\&\le \opnorm{Q_g^{\sigma,\ell+1}}_{\Lp2}\, 
      \opnorm{Q_g^{\sigma,\ell}}^a_{\Lp2}\,
      \|T_g^{\ell}\|^b_{\Ap2}\\
&\lesssim\opnorm{Q_g^{\sigma,\ell+1}}_{\Lp2}\,
  \|g\|_{\B^s_{\varphi}}^{s\ell a}\,\|g\|_{\B^1_{\varphi}}^{\ell b}
  \lesssim
  \opnorm{Q_g^{\sigma,\ell+1}}_{\Lp2}\,
 \|g\|_{\B^s_{\varphi}}^{s\ell a+\ell b},      
\end{align*}
where we have also  used the estimate
\eqref{eqn:upper:estimate:norm:Q}.
Since 
\begin{align}
\label{eqn:lem4:sharp:lower:estimate:norm:word4}
s\ell a+\ell b
&=s\ell(N-1)-\tfrac{s^2}{\sigma}+\tfrac{s}{\sigma}=s\ell(N-1)+(1-s)\tfrac{s}{\sigma}\\
\nonumber
&=s\ell(N-1)-s
=s\ell(N-\tfrac{\ell+1}{\ell}),
\end{align}
we deduce that
\begin{equation}
\label{eqn:lem4:sharp:lower:estimate:norm:word5}
A_1\lesssim\opnorm{Q_g^{\sigma,\ell+1}}_{\Lp2}\,
  \|g\|_{\B^s_{\varphi}}^{s\ell(N-\frac{\ell+1}{\ell})}.
\end{equation}

\noindent{\sf Estimate of $A_2$:}
\begin{align*}
&|T_g^{\ell+1}f| |g'| |g|^{N\sigma\ell-2}|T_g^{\ell}f|^{N-1}\left(1+\varphi' \right)^{-1}
\\&=(|g'|\left(1+\varphi' \right)^{-1})\,\bigl(|g|^{\sigma(\ell+1)}|T_g^{\ell+1}f|\bigr)\,
\bigl(|g|^{\sigma\ell}|T_g^{\ell}f|\bigr)^a\,|T_g^{\ell}f|^b
\\
&=(|g'|\left(1+\varphi' \right)^{-1})\,|Q_g^{\sigma,\ell+1}f|\,|Q_g^{\sigma,\ell}f|^a|T_g^{\ell}f|^b,
\end{align*}
where  $a=N-1-\frac{s+1}{\sigma\ell}$ and $b=\frac{s+1}{\sigma\ell}$. Since $1+a+b=N$, we may apply H\"{o}lder's inequality with exponents $p_1=N$, $p_2=\frac{N}a$ and $p_3=\frac{N}b$ and Theorem~{\ref{prop:radicality:estimateintro}}  to get 

\begin{align}
\label{eqn:lem4:sharp:lower:estimate:norm:word6}
 A_2
 &\lesssim\|g\|_{\B^1_{\varphi}}\,
   \|Q_g^{\sigma,\ell+1}f\|_{\Lp2}\,
      \|Q_g^{\sigma,\ell}f\|^a_{\Lp2}\,
      \|T_g^{\ell}f\|^b_{\Ap2} \\
\nonumber
&\lesssim \|g\|_{\B^s_{\varphi}}\,
\opnorm{Q_g^{\sigma,\ell+1}}_{\Lp2}\, 
\opnorm{Q_g^{\sigma,\ell}}^a_{\Lp2}\,
      \|T_g^{\ell}\|^b_{\Ap2}\\
\nonumber      
&\lesssim
\opnorm{Q_g^{\sigma,\ell+1}}_{\Lp2}\, 
\|g\|_{\B^s_{\varphi}}^{s\ell a+\ell b+1}
=\opnorm{Q_g^{\sigma,\ell+1}}_{\Lp2}\,
\|g\|_{\B^s_{\varphi}}^{s\ell(N-\frac{\ell+1}{\ell})},
\end{align}
where we have also used
\eqref{eqn:upper:estimate:norm:Q} and 
$s\ell a+\ell b+1=s\ell(N-\frac{\ell+1}{\ell})$.
\vspace*{6pt}

\noindent{\sf Estimate of $A_3$:}
\begin{equation*}\begin{split}
&|T_g^{\ell+1}f| |g'| |g|^{N\sigma\ell-1}|T_g^{\ell}f|^{N-2}|T_g^{\ell-1}f|\left(1+\varphi' \right)^{-1}
\\ &=(|g'|\left(1+\varphi' \right)^{-1})\,|Q_g^{\sigma,\ell+1}f|\,|Q_g^{\sigma,\ell}f|^a|Q_g^{\sigma,\ell-1}f||T_g^{\ell}f|^b,
\end{split}\end{equation*}
where  $a=N-2-\frac1{\sigma\ell}$ and $b=\frac1{\sigma\ell}$. Since $1+a+1+b=N$, we may apply H\"{o}lder's inequality with exponents $p_1=N$, $p_2=\frac{N}a$, $p_3=N$ and $p_4=\frac{N}b$   and Theorem~{\ref{prop:radicality:estimateintro}} to get 
\begin{align}
\label{eqn:lem4:sharp:lower:estimate:norm:word7}
 A_3
 &\le \|g\|_{\B^1_{\varphi}}\,
   \|Q_g^{\sigma,\ell+1}f\|_{\Lp2}\,
      \|Q_g^{\sigma,\ell}f\|^a_{\Lp2}\,
      \|Q_g^{\sigma,\ell-1}f\|_{\Lp2}\,
      \|T_g^{\ell}f\|^b_{\Ap2} \\
\nonumber
&\lesssim
\opnorm{Q_g^{\sigma,\ell+1}}_{\Lp2}\, 
\|g\|_{\B^s_{\varphi}}^{s\ell a+s(\ell-1)+\ell b+1} 
=\opnorm{Q_g^{\sigma,\ell+1}}_{\Lp2}\,
 \|g\|_{\B^s_{\varphi}}^{s\ell(N-\frac{\ell+1}{\ell})} ,
\end{align}
where we have also used 
\eqref{eqn:upper:estimate:norm:Q} and 
$s\ell a+s(\ell-1)+\ell b+1=s\ell(N-\frac{\ell+1}{\ell})$.\vspace*{6pt}

\noindent{\sf Estimate of $B$:}
\begin{align*}
B
&\lesssim\|\bigl((T_g^{\ell+1}f)\, g^{N\sigma\ell}\,(T_g^{\ell}f)^{N-2}\, (T_g^{\ell-1}f)   \bigr)'\|_{A^{\frac{p}N}_{W_{p,N} }} \\  
&\,\,\,\,  +\|T_g^{\ell+1}f\, \bigl(g^{N\sigma\ell}\,(T_g^{\ell}f)^{N-2}\,T_g^{\ell-1}f   \bigr)'\|_{A^{\frac{p}N}_{W_{p,N}}}  \\
&\lesssim\|(T_g^{\ell+1}f)\, g^{N\sigma\ell}\,(T_g^{\ell}f)^{N-2}\,T_g^{\ell-1}f   \|_{A^{\frac{p}N}_{\om_p}} \\ 
&\,\,\,\,   +\|(T_g^{\ell+1}f)\, g'g^{N\sigma\ell-1}(T_g^{\ell}f)^{N-2}\, T_g^{\ell-1}f   \|_{A^{\frac{p}N}_{W_{p,N}}}   \\
&\,\,\,  +\|(T_g^{\ell+1}f)\, g^{N\sigma\ell}g'(T_g^{\ell}f)^{N-3}(T_g^{\ell-1}f)^2   \|_{A^{\frac{p}N}_{W_{p,N}}} \\
&\,\,\,  +\|(T_g^{\ell+1}f)\, g^{N\sigma\ell}(T_g^{\ell}f)^{N-2}(T_g^{\ell-1}f)'   \|_{A^{\frac{p}N}_{W_{p,N}}} 
=B_1+A_3+B_2+B_3.
\end{align*}

{\sf Estimate of $B_1$:}
\begin{equation*}
|T_g^{\ell+1}f|\, |g|^{N\sigma\ell}\,|T_g^{\ell}f|^{N-2}\,|T_g^{\ell-1}f|
=|Q_g^{\sigma,\ell+1}f|\,
|Q_g^{\sigma,\ell}f|^{N-2}\,|Q_g^{\sigma,\ell-1}f|
\end{equation*}
By applying H\"{o}lder's inequality with exponents $p_1=N$, $p_2=\frac{N}{N-2}$ and $p_3=N$, and using estimate \eqref{eqn:upper:estimate:norm:Q}, we obtain that
\begin{align*}
 B_1
 &\le \|Q_g^{\sigma,\ell+1}f\|_{\Lp2}\, 
      \|Q_g^{\sigma,\ell}f\|^{N-2}_{\Lp2}\,
      \|Q_g^{\sigma,\ell-1}f\|_{\Lp2} \\
      \nonumber
&\le\opnorm{Q_g^{\sigma,\ell+1}}_{\Lp2}\, 
      \opnorm{Q_g^{\sigma,\ell}}^{N-2}_{\Lp2}\,
      \opnorm{Q_g^{\sigma,\ell-1}}_{\Lp2} \\
      \nonumber
&\lesssim  \opnorm{Q_g^{\sigma,\ell+1}}_{\Lp2}\, 
         \|g\|_{\B^s_{\varphi}}^{s\ell(N-\frac{\ell+1}{\ell})}.
\end{align*}

\noindent{\sf Estimate of $B_2$:}
\begin{equation*}\begin{split}
&|T_g^{\ell+1}f|  |g|^{N\sigma\ell}|g'||T_g^{\ell}f|^{N-3}|T_g^{\ell-1}f|^2\left(1+\varphi' \right)^{-1}
\\ &=(|g|^{\sigma}|g'|\left(1+\varphi' \right)^{-1})\,|Q_g^{\sigma,\ell+1}f|\,|Q_g^{\sigma,\ell}f|^{N-3}\,|Q_g^{\sigma,\ell-1}f|^2.
\end{split}\end{equation*}
We apply H\"{o}lder's inequality with exponents $p_1=N$, $p_2=\frac{N}{N-3}$ and $p_3=\frac{N}2$ and use \eqref{eqn:upper:estimate:norm:Q}  to get 
\begin{align}
\label{eqn:lem4:sharp:lower:estimate:norm:word10}
 B_2
&\lesssim \|g\|_{\B^s_{\varphi}}^s\,
\opnorm{Q_g^{\sigma,\ell+1}}_{\Lp2}\, 
      \opnorm{Q_g^{\sigma,\ell}}^{N-3}_{\Lp2}\,
      \opnorm{Q_g^{\sigma,\ell-1}}^2_{\Lp2}\\
\nonumber
&\lesssim \opnorm{Q_g^{\sigma,\ell+1}}_{\Lp2}\,
 \|g\|_{\B^s_{\varphi}}^{s\ell(N-\frac{\ell+1}{\ell})}.
\end{align}

\noindent{\sf Estimate of $B_3$ for $\ell=1$:}
\begin{align*}
&|T_g^{\ell+1}f|\,  |g|^{N\sigma\ell}\,|T_g^{\ell}f|^{N-2}\,|(T_g^{\ell-1}f)'|\left(1+\varphi' \right)^{-1}
\\&=|T_g^2f|\,|g|^{N\sigma}\,|T_gf|^{N-2}\,(|f'|\left(1+\varphi' \right)^{-1})\\
&=|Q_g^{\sigma,2}f|\,|Q_g^{\sigma,1}f|^{N-2}\,(|f'|\left(1+\varphi' \right)^{-1}).
\end{align*}
By applying H\"{o}lder's inequality with exponents $p_1=N$, $p_2=\frac{N}{N-2}$ and $p_3=N$ and using \eqref{eqn:upper:estimate:norm:Q},  we obtain that
\begin{align}
\label{eqn:lem4:sharp:lower:estimate:norm:word110}
 B_3
 &\lesssim \|Q_g^{\sigma,2}f\|_{\Lp2}\,
      \|Q_g^{\sigma,1}f\|^{N-2}_{\Lp2}\,
      \|f\|_{\Ap2} \\
\nonumber
&\lesssim
\opnorm{Q_g^{\sigma,2}}_{\Lp2}\, 
\|g\|_{\B^s_{\varphi}}^{s(N-2)} 
=\opnorm{Q_g^{\sigma,\ell+1}}_{\Lp2}\,
  \|g\|_{\B^s_{\varphi}}^{s\ell(N-\frac{\ell+1}{\ell})}.
\end{align}

\noindent{\sf Estimate of $B_3$ for $\ell>1$:}
\begin{align*}
&|T_g^{\ell+1}f|  |g|^{N\sigma\ell}|T_g^{\ell}f|^{N-2}|(T_g^{\ell-1}f)'|\left(1+\varphi' \right)^{-1}
\\&=|T_g^{\ell+1}f|  |g|^{N\sigma\ell}|T_g^{\ell}f|^{N-2}|g'||T_g^{\ell-2}f|\left(1+\varphi' \right)^{-1}\\
&=(|g|^{\sigma}|g'|\left(1+\varphi' \right)^{-1})|Q_g^{\sigma,\ell+1}f|\,|Q_g^{\sigma,\ell}f|^{N-2}\,|Q_g^{\sigma,\ell-2}f|.
\end{align*}
By applying H\"{o}lder's inequality with exponents $p_1=N$, $p_2=\frac{N}{N-2}$ and $p_3=N$,   and using   \eqref{eqn:upper:estimate:norm:Q}, we obtain that
\begin{align}
\label{eqn:lem4:sharp:lower:estimate:norm:word11}
 B_3
 &\le \|g\|_{\B^s_{\varphi}}^s\,
     \|Q_g^{\sigma,\ell+1}f\|_{\Lp2}\,
      \|Q_g^{\sigma,\ell}f\|^{N-2}_{\Lp2}\,
      \|Q_g^{\sigma,\ell-2}f\|_{\Lp2} \\
\nonumber
&\lesssim
\opnorm{Q_g^{\sigma,\ell+1}}_{\Lp2}\, \|g\|_{\B^s_{\varphi}}^{s\ell(N-2)+s(\ell-2)+s}
\\
\nonumber
&=\opnorm{Q_g^{\sigma,\ell+1}}_{\Lp2}\,
\|g\|_{\B^s_{\varphi}}^{s\ell(N-\frac{\ell+1}{\ell})}.
\end{align}
Finally, estimates \eqref{eqn:lem4:sharp:lower:estimate:norm:word5}-\eqref{eqn:lem4:sharp:lower:estimate:norm:word11} together with the hypotheses
\eqref{eqn:lem4:sharp:lower:estimate:norm:word1} show that
\begin{equation*}
\|g\|_{\B^s_{\varphi}}^{s\ell N}\lesssim\opnorm{Q^{\sigma,\ell}_g}_{\Lp2}^N
\lesssim
\opnorm{Q^{\sigma,\ell+1}_g}_{\Lp2}\,
\|g\|_{\B^s_{\varphi}}^{s\ell(N-\frac{\ell+1}{\ell})},
\end{equation*}
and hence estimate \eqref{eqn:lem4:sharp:lower:estimate:norm:word2} follows, which completes the proof of Proposition~{\ref{lem:sharp:lower:estimate:norm:word:4}}. 
\end{proof}

\subsection{Proof of Corollary~{\ref{cor:main:thm}}}
By Theorem~{\ref{thm:sharp:estimate:norm:word}}, if $L_{g,0}$ is bounded on $\Ap2$ then $g\in\B^{s_0}_{\varphi}$. On the other hand, since $s_j\le s_0$,
  Theorem~{\ref{prop:radicality:estimateintro}} shows that $g\in\B^{s_j}_{\varphi}$, for $j=1,\dots,J$. Then, by applying again Theorem~{\ref{thm:sharp:estimate:norm:word}}, we get 
\begin{equation*}
\|L_g\|_{\Ap2}
\le\sum_{j=0}^J\|L_{g,j}\|_{\Ap2}
\simeq\sum_{j=0}^J\|g\|_{\B^{s_j}_{\varphi}}^{N_j}
\lesssim\sum_{j=0}^J\|g\|_{\B^{s_0}_{\varphi}}^{N_j},
\end{equation*}
so $L_{g}$ is bounded on $\Ap2$.

Conversely, assume that $L_{g}$ is bounded on $\Ap2$. Then 
\begin{align*}
\|g_r\|_{\B^{s_0}_{\varphi}}^{N_0}\simeq\|L_{g_r,0}\|_{\Ap2}
&\lesssim\|L_{g_r}\|_{\Ap2}+\sum_{j=1}^J\|L_{g_r,j}\|_{\Ap2}
\\
&\lesssim\|L_{g_r}\|_{\Ap2}
+\sum_{j=1}^J\|g_r\|_{\B^{s_j}_{\varphi}}^{N_j}
\\
&\lesssim\|L_{g}\|_{\Ap2}
+\sum_{j=1}^J\|g_r\|_{\B^{s_0}_{\varphi}}^{N_j},
\end{align*} 
by Theorems~{\ref{thm:sharp:estimate:norm:word}-\ref{prop:radicality:estimateintro}} and Proposition~{\ref{prop:norm:dilations}}. 
Since $N_j<N_0$, for $j=1,\dots,J$, the above estimate shows that
$\sup_{0<r<1}\|g_r\|_{\B^{s_0}_{\varphi}}<\infty$, so $g\in\B^{s_0}_{\varphi}$, by \eqref{eqn:radialized:symbol:Bloch}, and therefore $L_{g,0}$ is bounded on $\Ap2$.

\section{Examples}\label{examples}

In the next result we provide simple conditions  on the function $\varphi$ which guarantee that the weight $\omega(z)=e^{-2\varphi(z)}$ lies in $\SW$.
\begin{proposition}\label{prop:example1}	
	Let $\varphi$ be a positive increasing $C^2$ function on $[0,1)$ satisfying the following conditions:
	\begin{enumerate}[label={\sf(\roman*)}, topsep=3pt, leftmargin=32pt,itemsep=3pt] 
		\item\label{item:111}  
$\lim_{r\to 1^-}\varphi(r)=\lim_{r\to 1^-}\varphi'(r)=\infty$.
		\item\label{item:113} $\varphi'$ is increasing, $\varphi'(0)=0$, and $\varphi''(0)>0$.
		\item\label{item:114}  There exists $\delta>0$ and a positive decreasing  $C^1$ function $\phi$ on $[0,1)$ so that $\tau(r)=(1+\varphi'(r))^{-\delta}\phi(r)$ satisfies 
		$\tau(r)^{-2}\simeq\varphi''(r)+\varphi'(r)$ and  $\lim_{r\to 1^-}\tau'(r)=
         \lim_{r\to 1^-}\tau'(r)\log\tau(r)=0$.
	\end{enumerate}  
	Then $\varphi$ extends to a radial function on $\D$ \textup{(}which we continue calling $\varphi$\textup{)} such that $\omega(z)=e^{-2\varphi(z)}$  is a smooth rapidly decreasing weight.
\end{proposition}

\begin{proof}
	Since $\varphi\in C^2_{\R}([0,1))$, its radial extension $\varphi$ is continuous on $\D$  and $C^2$ on $\D\setminus\{0\}$.
	Moreover,  the hypothesis $\varphi'(0)=0$ gives that  
	\begin{equation*}
		\lim_{z\to0}\frac{\partial\varphi}{\partial z}(z)=\lim_{z\to0}\frac{\partial\varphi}{\partial\overline{z}}(z)=
		\lim_{z\to0}\frac{\partial^2\varphi}{\partial z^2}(z)=\lim_{z\to0}\frac{\partial^2\varphi}{\partial\overline{z}^2}(z)=0	
	\end{equation*}
	and 
	\begin{equation*}
		\lim_{z\to0}\frac{\partial^2\varphi}{\partial z\partial\overline{z}}(z)=\tfrac12\varphi''(0).
	\end{equation*}
	Therefore $\varphi\in C^2(\D)$. Now $\Delta\varphi(z)=\varphi''(|z|)+\frac1{|z|}\varphi'(|z|)>0$, for $z\in\D\setminus\{0\}$,
	 and $\Delta\varphi(0)=2\varphi''(0)$  (by \ref{item:113}), so \ref{rapidly:decreasing:condition:a}  holds. 
    Moreover, since $1+\varphi'$ is a positive increasing function on $[0,1)$, $\tau$ is decreasing on $[0,1)$, so  
	 \ref{rapidly:decreasing:condition:b}  and \ref{rapidly:decreasing:condition:c} directly follow from  \ref{item:114}.
	Finally, $\varphi$ is convex (because $\varphi'$ is increasing) and $(1+\varphi')\tau^{1/\delta}=\phi^{1/\delta}$ is decreasing, and therefore \ref{rapidly:decreasing:condition:e} holds.
	And that ends the proof.
\end{proof}

\begin{proposition}\label{prop:example2}
	Let $\varphi$ be a positive increasing $C^2$ function on $[0,1)$ satisfying \eqref{varphi} and conditions 
    \ref{item:111} and \ref{item:113} in the statement of Proposition \ref{prop:example1}.
	Then the function $\psi=e^\varphi$ also satisfies \eqref{varphi} and all the hypotheses of Proposition \ref{prop:example1}.
\end{proposition}

\begin{proof}
	It is clear that $\psi$ is a positive increasing $C^2$ function on $[0,1)$ such that $\psi(r)\to\infty$, as $r\to1^{-}$. Moreover, $\psi'=\varphi'\psi$, $\psi''=\psi(\varphi''+(\varphi')^2)$, and, in particular,
    $\psi'$ is increasing, $\psi'(r)\to\infty$, as $r\to1^{-}$,
	$\psi'(0)=0$ and  $\psi''(0)=\psi(0)\varphi''(0)>0$, so $\psi$ satisfies \ref{item:111} and \ref{item:113}.
	
	By \eqref{varphi} and the fact that $\varphi$ is a positive increasing function, we have that $\varphi''(r)+\varphi'(r)^2\lesssim (1+\varphi'(r))^2$. On the other hand, by \ref{item:111} and  \ref{item:113}, 
	$(1+\varphi'(r))^2\simeq 1+\varphi'(r)^2\lesssim\varphi''(r)+\varphi'(r)^2$. Therefore 
\begin{equation}\label{eqn:pre:laplacian:varphi}
\varphi''(r)+\varphi'(r)^2\simeq (1+\varphi'(r))^2,
\end{equation}
and, in particular, 
\begin{equation*}
\psi''(r)+\psi'(r)\simeq\psi(r)(1+\varphi'(r))^2
\qquad\mbox{and}\qquad
\psi''(r)\simeq \psi(r)(1+\varphi'(r))^2.
\end{equation*}
Now the estimate 
\begin{equation}\label{eqn:estimate:1:plus:psiprime}
1+\psi'(r)\simeq\psi(r)(1+\varphi'(r))
\end{equation}
shows that 
\begin{equation}\label{eqn:estimate:psidoubleprime}
\psi''(r)+\psi'(r)\simeq\psi''(r)\simeq(1+\psi'(r))(1+\varphi'(r)),
\end{equation} and so it is clear that
$\tau_{\psi}(r):=(1+\psi'(r))^{-1/2}(1+\varphi'(r))^{-1/2}$
 satisfies
$\tau_{\psi}(r)^{-2}\simeq\psi''(r)+\psi'(r)$. Note that $(1+\varphi')^{-1/2}$ is a positive decreasing  $C^1$ function on $[0,1)$.
Thus in order to show that $\psi$ satisfies \ref{item:114} we only have to check that
\begin{align}
\label{eqn:lim:tauprime}
&\lim_{r\to 1^-}\tau_{\psi}'(r)=0\qquad\mbox{and}
\\
\label{eqn:lim:tauprime:log}
&\lim_{r\to 1^-}\tau_{\psi}'(r)\log\tau_{\psi}(r)=0.
\end{align}
Now \eqref{eqn:estimate:psidoubleprime},
 \eqref{eqn:estimate:1:plus:psiprime}, and \eqref{eqn:pre:laplacian:varphi}  show that
\begin{align*}
-2\tau'_{\psi}(r)
&=\frac{\psi''(r)}{(1+\psi'(r))^{3/2}(1+\varphi'(r))^{1/2}}
  +\frac{\varphi''(r)}{(1+\psi'(r))^{1/2}(1+\varphi'(r))^{3/2}}
\\
&\simeq
 (1+\varphi'(r))^{1/2}(1+\psi'(r))^{-1/2}
 +\varphi''(r)\psi(r)^{-1/2}(1+\varphi'(r))^{-2}
\\
&\lesssim\psi(r)^{-1/2},
\end{align*}
which clearly implies \eqref{eqn:lim:tauprime}.
%
On the other hand, \eqref{eqn:estimate:1:plus:psiprime} implies that
\begin{align*}
-2\log\tau_{\psi}(r)
&=\log(1+\psi'(r))+\log(1+\varphi'(r))
\\
&\lesssim\log\psi(r)+\log(1+\varphi'(r)),
\end{align*}
and so, by the preceding estimate, we obtain that
\begin{equation*}
|\tau_{\psi}'(r)\log\tau_{\psi}(r)|
\lesssim\psi(r)^{-1/2}\log\psi(r)+e^{-\frac12\varphi(r)}\log(1+\varphi'(r)).
\end{equation*}
It is clear that $\lim_{r\to1^{-}}\psi(r)^{-1/2}\log\psi(r)=0$.
Moreover, L'H\^{o}pital's rule gives that 
\begin{equation*}
\lim_{r\to1^{-}}e^{-\frac12\varphi(r)}\log(1+\varphi'(r))=
\lim_{r\to1^{-}}\frac{2\varphi''(r)}{\varphi'(r)(1+\varphi'(r))e^{\frac12\varphi(r)}}=0,
\end{equation*}
since \eqref{varphi} gives the estimate
\begin{equation*}
\frac{\varphi''(r)}{\varphi'(r)(1+\varphi'(r))e^{\frac12\varphi(r)}}
\lesssim\frac1{\varphi(r)e^{\frac12\varphi(r)}}
\qquad(\tfrac12<r<1).
\end{equation*}
Hence \eqref{eqn:lim:tauprime:log} holds.
Finally, $\psi$ satisfies \eqref{varphi} because
\begin{equation*}
\frac{\psi''(r)\psi(r)}{(1+\psi'(r))^2}=
\frac{\psi(r)^2\bigl(\varphi''(r)+\varphi'(r)^2\bigr)}{(1+\psi'(r))^2}\stackrel{\mbox{\tiny$(\ast)$}}{\simeq}
\frac{\varphi''(r)+\varphi'(r)^2}{(1+\varphi'(r))^2}\stackrel{\mbox{\tiny$(\star)$}}{\simeq}1,
\end{equation*}
where $(\ast)$ and $(\star)$ follow from the estimates \eqref{eqn:estimate:1:plus:psiprime} and \eqref{eqn:pre:laplacian:varphi}, respectively.
\end{proof}

Since, for any $\alpha,c>0$,
$\varphi_{\alpha,c}(r)=\frac{c}{(1-r^2)^{\alpha}}$
is a positive increasing $C^2$ function on $[0,1)$ satisfying \eqref{varphi} and conditions 
    \ref{item:111} and \ref{item:113} in the statement of Proposition \ref{prop:example1}, Propositions \ref{prop:example1} and \ref{prop:example2}	together with a straightforward induction argument show the following corollary, which gives weights in $\SW$ decreasing to $0$ 
''exponentially'' as fast as you want.

\begin{corollary}\label{cor:examples}
For any $n\in \N_0$ the radial weight $\omega_n$ defined by \eqref{eqn:examples:intro} belongs to $\SW$. 
\end{corollary}

\end{document}